%% file: main.tex
\documentclass[11pt,letterpaper]{amsart}
\input{head}

\newcommand{\R}{\mathbf R}
\newcommand{\B}{\mathbf B}
\newcommand{\Sph}{\mathbf S}
\newcommand{\cA}{\mathcal A}
\newcommand{\BiRic}{\operatorname{BiRic}}

\newcommand{\p}{\partial}

\newcommand{\NN}{\mathbb{N}}

\newcommand{\sff}{\mathrm{I\!I}}
\newcommand{\dist}{\operatorname{dist}}

\DeclareMathOperator{\dive}{div}
\DeclareMathOperator{\capa}{Cap}

\makeatletter
\@namedef{subjclassname@2020}{\textup{2020} Mathematics Subject Classification}
\makeatother

\begin{document}

\title[Stable FBMH]{Stable free boundary minimal hypersurfaces in $\mathbb{B}^5$ and $\mathbb{B}^6$ }

\author{Han Hong}
\address{Department of Mathematics and statistics \\ Beijing Jiaotong University \\ Beijing \\ China, 100044}
\email{hanhong@bjtu.edu.cn}

\author{Yujie Wu}
\address{Institut f\"{u}r Mathematik\\Universit\"{a}t Potsdam\\Potsdam\\ Germany, 14476}
\email{yujiewu2025@gmail.com}

\begin{abstract}
In this paper, we prove that there is no complete two-sided stable immersed free boundary minimal
hypersurface in the closed Euclidean unit ball $\mathbb{B}^{n+1}$ for
$3\leq n\leq5$. 
\end{abstract}

\maketitle

\section{Introduction}
The rigidity and nonexistence of complete stable minimal hypersurfaces have been important topics in geometric analysis. A complete two-sided minimal immersion
\[
M^n\longrightarrow (X^{n+1},g)
\]
is stable if
\[
\int_M |\nabla\phi|^2
\geq
\int_M\bigl(|A|^2+\operatorname{Ric}_X(\nu,\nu)\bigr)\phi^2
\]
for every compactly supported test function $\phi$, where $A$ is the second fundamental form of $M$ and $\nu$ is a globally defined unit normal. There are several well-known results. First, nonnegative Ricci curvature forces a closed stable minimal hypersurface to be totally geodesic with vanishing normal Ricci curvature, while positive Ricci curvature rules out such hypersurfaces. Second, nonnegative scalar curvature strongly restricts the topology and conformal type of stable minimal surfaces (sphere, torus, plane and cylinder); see \cite{Fischer-Colbrie-Schoen-The-structure-of-complete-stable}. A uniform positive scalar curvature lower bound also gives a diameter bound and hence compactness of the stable minimal surface; see \cite{Schoen-yau-79a,Schoen-Yau-PSC}. The main idea is that a stable minimal surface inherits positive Gaussian curvature in a spectral sense. Shen and Ye \cite{shenyingyerugang} extended this idea and obtained diameter estimates for stable minimal hypersurfaces when the ambient manifold has uniformly positive bi-Ricci curvature and $n\leq4$. More recently, Catino, Mari, Mastrolia, and Roncoroni \cite{catino-Mari-Mastrolia-Roncoroni} treated the case $n=5$ under uniformly positive bi-Ricci curvature and nonnegative $4$-intermediate Ricci curvature. Mazurowski, Wang, and Yao \cite{liam-wang-yao} obtained the same result in which the latter assumption is replaced by a Ricci curvature bound. These results imply that there is no complete noncompact two-sided minimal hypersurface in the unit sphere $\mathbb{S}^{n+1}$ for $n\leq 5.$

When $M$ is noncompact and $X$ is Euclidean space, the corresponding rigidity question is the stable Bernstein problem. It asks whether every complete two-sided stable minimal hypersurface in $\mathbb{R}^{n+1}$ for $n\leq 6$ must be a hyperplane. The classical case of surfaces in $\mathbb{R}^3$ was settled in the early 1980s \cite{Fischer-Colbrie-Schoen-The-structure-of-complete-stable,doCarmo-Peng,Pogorelov-stable}. Recent work has made major progress in higher dimensions; see \cite{chodoshliR4,catino-Mari-Mastrolia-Roncoroni,chodoshliR4anisotropic,Chodosh-Li-Minter-Stryker,mazet,cabre2026gradientestimatesgreenkernel}. These results have also led to rigidity and nonexistence theorems  for complete noncompact stable minimal and constant mean curvature hypersurfaces in curved ambient manifolds; see \cite{chodosh-li-stryker,Hong24,chen-hong-li-docarmoquestion,miranda2026finiteindexconstantmean}.

For manifolds with boundary, the natural objects are free boundary minimal hypersurfaces. Throughout this paper, completeness means completeness of the induced metric as a metric space, including the boundary. Let $(X^{n+1},\partial X)$ be a Riemannian manifold with boundary. A complete free boundary minimal immersion
\[
F:(M^n,\partial M)\longrightarrow (X^{n+1},\partial X)
\]
is minimal in the interior and meets $\partial X$ orthogonally along $\partial M$. Its stability inequality has an additional term involving the second fundamental form of $\partial X$:
\[
\int_M |\nabla\phi|^2
\geq
\int_M\bigl(|A|^2+\operatorname{Ric}_X(\nu,\nu)\bigr)\phi^2
+\int_{\partial M}h_{\partial X}(\nu,\nu)\phi^2.
\]
 
 Carlotto and Franz \cite{carlotto-franz} proved a diameter estimate for $M$ when $n=2$ under either of the following assumptions: $R_X\geq c>0$, $H_{\partial X}\geq0$; or $R_X\geq0$ and $H_{\partial X}\geq H_0>0$. In particular, every complete stable free boundary minimal surface satisfying these assumptions is compact. Their proof uses a conformal change under which $\partial M$ becomes convex, followed by the theory of complete surfaces with nonnegative curvature. This argument does not extend directly to higher dimensions. Notice that the second set of assumptions includes the unit ball $\mathbb{B}^3\subset\mathbb{R}^3$.

Inspired by \cite{chodosh-li-stryker}, the second author proved rigidity and nonexistence results in dimension $n=3$ under suitable assumptions on intermediate Ricci curvature, scalar curvature, and boundary convexity \cite{wuyujie,wuyujiennsc}. In particular, one has the following theorem.

\begin{theorem}[{\cite{wuyujiennsc}}]\label{wu's theorem}
There is no complete stable free boundary minimal immersion in the closed Euclidean unit ball $\mathbb{B}^4$.
\end{theorem}

It is useful to recall two possible approaches in proving nonexistence of noncompact stable minimal hypersurfaces. The first is to prove a diameter estimate directly as in \cite{Schoen-yau-79a,shenyingyerugang,catino-Mari-Mastrolia-Roncoroni} as mentioned before. The second is to prove geometric rigidity under nonnegative curvature assumptions and then use strict positivity to obtain nonexistence. In the second approach, one needs suitable volume growth in order to use cutoff functions in the stability inequality. In particular, at most quadratic volume growth is enough to get vanishing second fundamental form. This approach was developed in \cite{Fischer-Colbrie-Schoen-The-structure-of-complete-stable,chodosh-li-stryker,hong-yan}.

For free boundary hypersurfaces in a Euclidean ball, a direct diameter estimate is not available in higher dimensions because the interior ambient curvature vanishes. One can obtain only an inradius bound; see Appendix \ref{appendix}. Such a bound does not imply compactness because $\partial M$ may be noncompact. In \cite{wuyujiennsc}, the second author instead adapted the volume-growth approach of \cite{chodosh-li-stryker}. The nonparabolic end was divided into annular blocks of fixed width, and a stable warped $\theta$-bubble was constructed in each block. The capillary estimates give a separating disk with uniformly bounded boundary length and inradius, and hence uniformly bounded diameter. This gives a uniform diameter bound for each annular block. Curvature and local volume estimates then imply a uniform volume bound for the blocks and therefore almost linear volume growth of the hypersurface. This argument uses a diameter estimate for a two-dimensional warped $\theta$-bubble and does not directly apply when the bubbles have higher dimension.

In this paper, we use a different method to prove the following result.

\begin{theorem}\label{MainThem}
There is no complete stable immersed free boundary minimal hypersurface in $\mathbb{B}^{n+1}$ for $3\leq n\leq5$.
\end{theorem}

For a free boundary minimal hypersurface in $\mathbb{B}^{n+1}$, stability takes the form
\begin{equation}\label{eq:fb-stability}
\int_M |\nabla\varphi|^2\,d\mu
\geq
\int_M |A|^2\varphi^2\,d\mu
+\int_{\partial M}\varphi^2\,d\sigma
\end{equation}
for every compactly supported test function $\varphi$. If $M$ is compact, taking $\varphi\equiv1$ immediately gives a contradiction. Thus it remains to rule out complete noncompact, and necessarily nonproper, examples. The case $n=2$ follows from \cite{carlotto-franz} and $n=3$ follows from Theorem \ref{wu's theorem}, so the new part of the proof concerns $n=4$ and $n=5$.

We now describe the proof. After passing to the universal cover, we first study the ends of $M$. The harmonic-function argument from \cite{wuyujie,wuyujiennsc}, together with the infinite-volume property of a complete noncompact free boundary minimal hypersurface in a bounded Euclidean domain, shows that $M$ has only one end, see Theorem \ref{OneEnd}. This end meets $\partial M$ arbitrarily far out and is nonparabolic. This step gives the topological structure needed to construct connected bands that exhaust the end.

Next, set $r=|F|$ and consider the conformal metric
\[
\widetilde g=r^{-2}g
\]
on $N=M\setminus F^{-1}(0)$. The punctured Euclidean ball becomes the half cylinder $\mathbb{S}^n\times[0,\infty)$, and its boundary is totally geodesic. The original stability inequality gives a spectral bi-Ricci inequality on $(N,\widetilde g)$ with a Robin boundary term. For $n=4$ we use the usual bi-Ricci curvature in \cite{Chodosh-Li-Minter-Stryker}, while for $n=5$ we use a weighted bi-Ricci curvature and the choice of parameters from the stable Bernstein argument in \cite{mazet}. Most of the computations in this part have been conducted in \cite{Chodosh-Li-Minter-Stryker,mazet}.

On every sufficiently wide band in the unique end, we then construct a free boundary weighted $\mu$-bubble with respect to the $\tilde{g}$-metric. Combining its second variation with the spectral bi-Ricci inequality gives a positive spectral Ricci lower bound on the bubble, together with a useful boundary term. A boundary spectral comparison theorem gives a uniform volume bound for every separating bubble. Since $r\leq1$, the same uniform upper bound holds for the bubble in the original metric. We remark here that directly constructing free boundary weighted $\mu$-bubble in the original metric would not give a uniform bound on the volume of $\mu$-bubble.

The final step converts the bubble volume bound into a volume bound for the annular blocks of $M$. Curvature estimates for stable free boundary minimal hypersurfaces give a uniform lower bound for $\operatorname{Ric}_M$. The Jacobian comparison theorem along minimizing geodesics is used to bound the volume of the block by a fixed multiple of the volume of the bubble. Therefore the annular blocks have uniformly bounded volume, and the nonparabolic end of \(M\) has almost linear volume growth.

Finally, almost linear volume growth provides compactly supported cutoff functions $\varphi_k$ such that $\varphi_k\to1$ locally and
\[
\int_M|\nabla\varphi_k|^2\,d\mu\longrightarrow0.
\]
Substituting these functions into \eqref{eq:fb-stability} forces the interior and boundary terms on the right-hand side to vanish. This contradicts the positive boundary term of the Euclidean unit ball and completes the proof of Theorem \ref{MainThem}.

\vskip.2cm
The rest of the paper is organized as follows. In Section 2, we study the
ends of complete stable free boundary minimal hypersurfaces and prove the
existence and uniqueness of a nonparabolic boundary end. In Section 3, we recall the Gromov--Lawson conformal metric and derive the spectral weighted
bi-Ricci curvature inequalities. In Section 4, we construct free boundary
weighted $\mu$-bubbles and establish the required volume estimates. In Section 5, we prove almost linear volume growth and complete the proof of the main theorem. The appendix contains an inradius estimate for stable free boundary minimal
hypersurfaces.

\subsection{Ackownledgement}We thank Prof. Otis Chodosh for helpful discussions and encouragement. The first author is supported by the Fundamental Research Funds for the Central Universities No. YA26JBMC00040 and NSFC No. 12401058. 
The second author is funded by the European Union (ERC Starting Grant 101116001–COMSCAL)\footnote{Views and opinions expressed are however those of the author(s) only and do not necessarily reflect those of the European Union or the European Research Council. Neither the European Union nor the granting authority can be held responsible for them.}. 
  The second author also thanks the Academy of Mathematics and Systems Science, Chinese Academy of Sciences for their hospitality while part of this work was conducted.

\section{One Boundary Nonparabolic End}
We prove in this section that a complete two-sided stable free boundary  minimal hypersurface in a bounded convex domain in Euclidean space is compact if and only if its boundary is compact. If it is non-compact, it can only have one end, which is a boundary nonparabolic end.
For references on ends and capacity, see for example \cite{grigor1990dimension} and \cite{li2004lectures}.


Given a sequence of compact sets \(K_i\) exhausting \(M\), and an end of \(M\) is a sequence of connected unbounded components \(\{U_i\}_{i\in \NN}\) consisting of a nested sequence \(U_{i+1}\subset U_{i}\subset (M\setminus K_{i}) \). We say that \(\{U_i\}_{i\in \NN}\) and \(\{V_i\}_{i\in \NN}\) represent the same class of end if they are eventually contained in one another, that is, for every \(i\in \NN\) there is \(j,k\in \NN\) such that \(U_j \subset V_i\) and \(V_k \subset U_i\).
We first recall the notion of ends of a complete non-compact manifold $M$ with boundary. 

\begin{definition}[Boundary and interior ends]
An end $E=(E_j)_{j\in \NN}$ is a \emph{boundary end} if every $E_j$ satisfies
\[
        E_j\cap\partial M\neq\varnothing.
\]
It is an \emph{interior end} if there is some \(j\in \NN\) with
\[
        E_j\cap\partial M=\varnothing.
\]
Thus an interior end is eventually contained in $\Int M$, whereas a boundary
end continues to meet $\partial M$ arbitrarily far out. This definition does not depend on the choice of representative for a class of end.
\end{definition}
For the present free-boundary problem, a boundary end may remain in the compact sphere $\Sph^n$ and wind around there
indefinitely.

Let $E$ be a representative of an end with respect to a compact set $K$, we decompose its boundary as
\[
        \partial E=\partial_0E\cup\partial_1E,
        \qquad \partial_0E=\partial E\cap\partial M,
\]
where $\partial_1E$ is the inner boundary contained in $\partial K$.
After an arbitrarily small perturbation of \(K\), we may assume that \(\partial_0 E\) and \(\partial_1E\)
meet transversely with constant angle in $(0,\pi/4)$; by Theorem 4.5 in \cite{wuyujie}, solutions to the mixed boundary value problem (Neumann on \(\p_0 E\) and Dirichlet on \(\p_1 E\)) are at least \(C^{2,\alpha(\theta)}\) for some \(\alpha>0\) depending on the intersection angle \(\theta\in (0,\pi/4)\), this is the mixed
boundary regularity used below.  

\subsection{Capacity and nonparabolic ends}
\begin{definition}
    Let \(M\) be a complete non-compact Riemannian manifold with boundary (could be empty) and \(E=(E_j)_{j\in \NN}\) an end of \(M\), we define the capacity of \(E_j\) as,
    \begin{align*}
        \capa(E_j)=\inf_{\varphi \in C^1_c(E_j), 0\leq \varphi\leq 1,\varphi\rvert_{\partial_1 E_j}=1} \int_{E_j}|\nabla \varphi|^2.
    \end{align*}
    If for some \(j_0 \in \NN, \capa(E_{j_0})>0\), then \(\capa(E_j)>0\) for all \(j\in\NN\)  (see for example Lemmas 5.5 and 5.6 in \cite{wuyujiennsc}).
    We say an end \((E_j)_{j\in \NN}\) is nonparabolic if \(\capa(E_{j_0})>0\) for some \(j_0\in \NN\). Otherwise we say that the end \((E_j)_{j\in\NN}\) is parabolic.
\end{definition}

We recall the following lemma whose proof is given for completeness.
\begin{lemma}\label{capa_K}
    If \(E\) is a connected unbounded component of \(M\setminus K'\) for a compact subset \(K'\), for a compact subset \(K\) of \(E\) define, 
    \begin{align*}
        \capa_E(K):=\inf_{\varphi\rvert_K=1, \varphi\in C^{1}_c(E), 0\leq \varphi \leq 1} \int_{E}|\nabla \varphi|^2.
    \end{align*}
    We show the following,
    \begin{itemize}
        \item if \(\capa(E)=0\) then for any compact subset \(K\) of \(E\), \(\capa_E(K)=0\);
        \item if \(\capa_E(K)=0\) for some compact \(K\) with nonempty interior of \(E\), then \(\capa(E)=0\).
    \end{itemize}
\end{lemma}

\begin{proof}
Choose a smooth exhaustion $\{\Omega_i\}_{i=1}^{\infty}$ of $E$ such
that
\[
\partial_1E\subset\partial\Omega_i,
\qquad
\Omega_i\subset\Omega_{i+1},
\qquad
\bigcup_{i=1}^{\infty}\Omega_i=E.
\]

Let $u_i$ be the minimizer of the Dirichlet energy on $\Omega_i$ subject to
\[
u_i=1 \text{ on }\partial_1E,
\quad
u_i=0 \text{ on }\Sigma_i=
\partial\Omega_i\setminus
\bigl(\partial_1E\cup\partial_0E\bigr), \quad \partial_\nu u_i=0. \text{ on } \p_0 E.
\]

Thus $u_i$ satisfies
\[
\begin{cases}
\Delta u_i=0 & \text{in }\Omega_i,\\
u_i=1 & \text{on }\partial_1E,\\
u_i=0 & \text{on }\Sigma_i,\\
\partial_\nu u_i=0 & \text{on }\partial_0E\cap\partial\Omega_i.
\end{cases}
\]
By the maximum principle,
\[
0\leq u_i\leq1.
\]
Since harmonic functions minimize Dirichlet energy while fixing the boundary value,
\[
\int_{\Omega_i}|\nabla u_i|^2
\rightarrow
\operatorname{Cap}(E).
\]
If \(\capa(E)=0\), comparing the boundary values of \(u_i\) and using the maximum principle , the sequence $\{u_i\}$ is increasing. We can first extract a subsequence converging weakly in \(W^{1,2}\), then use standard elliptic regularity.  We denote $u$ as its locally uniform limit in \(C^{2,\alpha}_{\text{loc}}\). The lower semicontinuity of the
Dirichlet energy gives
\[
\int_E|\nabla u|^2=0.
\]
Since $E$ is connected, $u$ is constant and \(u \rvert_{\p_1 E}=1\).
Consequently,
\[
u_i\longrightarrow1
\qquad\text{locally uniformly on }E.
\]

Now let $K\subset E$ be compact. For \(i\geq 2\) set
\[
a_i
=
\left(1-\frac1i\right)\min_Ku_i.
\]
Then
\[
0<a_i<\min_Ku_i,
\qquad
a_i\longrightarrow1.
\]
Define the Lipschitz function,
\[
v_i
=
\min\left\{1,\frac{u_i}{a_i}\right\}.
\]
Since $a_i<\min_Ku_i$, the continuity of $u_i$ implies that
$v_i\equiv1$ on a neighborhood of $K$. Extending $v_i$ by zero outside
$\Omega_i$.
Then
\[
\int_E|\nabla v_i|^2
\leq
\frac{1}{a_i^2}
\int_{\Omega_i}|\nabla u_i|^2
\longrightarrow0.
\]
Since $v_i\equiv1$ on a neighborhood of $K$, a standard smoothing
argument gives functions $\widetilde v_i\in C_c^1(E)$ such that
\[
0\leq\widetilde v_i\leq1,
\qquad
\widetilde v_i|_K=1,
\qquad
\int_E|\nabla\widetilde v_i|^2\longrightarrow0.
\]
Therefore,
\[
\operatorname{Cap}_E(K)=0.
\]

Conversely, suppose that
\[
\operatorname{Cap}_E(K)=0
\]
for some compact $K\subset E$ with nonempty interior, by definition there is \(h_j\) compactly supported, \(h_j \rvert_{K}=1\) and the Dirichlet energy of \(h_j\) converges to 0. Using the same argument as in the previous case, we have \(h_j\) converges to \(1\) locally uniformly on \(E\). Now take any compact subset \(K'\) of \(E\) and take \(j\geq j_0\) so that \(h_j \rvert_{K'}\geq \frac{1}{2}\), define \(v_j=\min\{1,2h_j\}\), then
\[
\capa_E(K')\leq \int_E |\nabla v_j|^2 \leq 4\int_{E}|\nabla h_j|^2\rightarrow 0.
\]
This implies \(\capa_E(K')=0\) for any compact \(K'\) and in particular \(\capa(E)=0\).
\end{proof}

Cao-Shen-Zhu \cite{Cao-Shen-Zhu-infinitevolume} proved that interior ends of minimal hypersurface are nonparabolic using the Michael-Simon-Sobolev inequality.
\begin{lemma}
    If \((M^n,\partial M)\rightarrow \mathbb{R}^{n+1}\), $n\geq 3$, is a complete non-compact minimal hypersurface, and \((E_j)_{j\in \NN}\) is an interior end, then \((E_j)_{j\in\NN}\) is nonparabolic.
\end{lemma}
\begin{proof}
    We take \(j\) large enough so that \(E_j \cap \partial M =\emptyset\).
    The Sobolev inequality for minimal hypersurfaces in \(\mathbb{R}^{n+1}\) gives for any Lipschitz domain \(D \subset E_j\), and any Lipschitz function \(\phi\)  that vanishes on \(\partial D\),
    \begin{align*}
        \left( \int_{D}\phi^p\right)^{\frac{2}{p}} \leq c_n \int_D |\nabla \phi|^2,
    \end{align*}
    where $p=2n/(n-2)$ and \(c_n\) is the Sobolev constant that only depends on the dimension.
    
    We apply the above Michael-Simon-Sobolev inequality to a Lipschitz function  \(\phi\) whose support \(D\) is in \(E_j\). Let  \(\phi \rvert_{K}=1\) for some compact subset \(K\) in \(D\) with non-empty interior, we have
    \begin{align*}
        \int_{E_{j}} |\nabla \phi|^2=\int_{D} |\nabla \phi|^2\geq c_n^{-1}\left( \int_{K}\phi^p\right)^{\frac{2}{p}} >c_n^{-1}\Vol^{\frac{2}{p}}(K)>0.
    \end{align*} 
By Lemma \ref{capa_K}, the end \((E_j)_{j\in \NN}\) is nonparabolic.
\end{proof}

Using the stability inequality, we can prove that boundary ends of stable free boundary minimal hypersurfaces are also nonparabolic.

\begin{lemma}
    If \((M^n,\partial M) \rightarrow (X^{n+1},\partial X)\) is a two-sided stable free boundary minimal hypersurface, assume \(\Ric_X \geq 0, \sff_{\partial X}\geq 1\), and \((E_j)_{j\in \NN}\) is a boundary end, then \(\capa(E_j)>0\) for any \(j\in \NN\), i.e. \((E_j)_{j\in\NN}\) is nonparabolic.
\end{lemma}

\begin{proof}
    By the stability inequality, we have for any compactly supported smooth function \(\phi\) on \(M\),
    \begin{align*}
        \int_{M} |\nabla \phi|^2\geq& \int_{M}(\Ric(\nu_M,\nu_M)+|\sff_M|^2)\phi^2+\int_{\partial M}\sff_{\partial X}(\nu_M,\nu_M)\phi^2\\
        \geq &\int_{\partial M}\phi^2.
    \end{align*}
    Since \(E_0 \cap \partial M \neq \emptyset\), take a subset \(K'\) of non-empty interior of \(E_0 \cap \partial M \neq \emptyset\).
    Now take a collar neighborhood of \(K'\) inside \(E_0\), denoted as \(K\), we may require \(K\cap \p_1 E_0=\emptyset\). We take a compactly supported function \(\phi \in C^1_{c}(E_0)\) and \(\phi\rvert_{K}=1\), then 
    \begin{align*}
        \int_{E_0}|\nabla\phi|^2=\int_M |\nabla \phi|^2 \geq \Vol(K')>0.
    \end{align*}
    By Lemma \ref{capa_K}, this implies the end \((E_j)_{j\in \NN}\) is nonparabolic.
\end{proof}

\subsection{Harmonic function and nonparabolic ends}
We will extend the results on ends of free boundary minimal hypersurfaces in \cite{wuyujie} to higher dimensions. 
Wu only discussed the case when $n+1=4$, the proof can be extended to all dimensions with different curvature assumptions.

\begin{definition}
    Let \(X^{n+1}\) be a Riemannian manifold with or without boundary, denote \(K(\cdot, \cdot)\) to be the sectional curvature, for any orthonormal $n$-frame \(e_1,...,e_n\) in a tangent space $T_pM$, define its \((n-1)\)-intermediate curvature at $p$ to be,
    \[
    \Ric_{n-1}(e_1,...,e_n)=\sum_{i=1}^{n-1} K(e_i,e_n).
    \]
    We say \(X^{n+1}\) has \(\Ric_{n-1}\geq 0\) if \(\Ric_{n-1}(e_1,...,e_n)\geq 0\) for any orthonormal $n$-frame and any $p\in X.$.
\end{definition}

\begin{definition}
    Given a Riemannian manifold $M$ with boundary, we denote \(A^{\p X}\) as the second fundamental form of \(\p X\). We say \(\p X\) is two-convex, denoted as \(A^{\p X}_2 \geq 0\), if for any orthonormal pair \(e_1, e_2\) tangent to $\partial X$, \(A^{\p X}(e_1,e_1)+A^{\p X}(e_2,e_2)\geq 0\).
\end{definition}

The following theorem extends \cite[Theorem 5.3]{wuyujie}
to higher dimensions under the corresponding intermediate
Ricci curvature assumption.
\begin{theorem}\label{one nonparabolic end}
Let $(X^{n+1},\partial X)$ have $\Ric_{n-1}^X\geq0$ and $A^{\p X}_2\geq0$, and let
$(M^n,\partial M)\hookrightarrow(X,\partial X)$ be a complete, two-sided,
stable, free-boundary minimal immersion  of infinite volume.  Then $M$ has at most
one nonparabolic end.
\end{theorem}

To prove this theorem, we first show the following proposition.
\begin{proposition}[{\cite[Theorem 5.2]{wuyujie}}]\label{thm:wu-arbitrary-dim}
Let $(X^{n+1},\partial X)$ be complete with
$\Ric_{n-1}^X\geq0$, and let
$(M^n,\partial M)\hookrightarrow(X,\partial X)$ be a two-sided stable
free-boundary minimal immersion.  Let $u$ be harmonic on $M$ and satisfy
$\partial_\nu u=0$ on $\partial M$.  Then, for every
compactly supported smooth function $\phi$,
\begin{align}
 &\frac1n\int_M \phi^2|A_M|^2|\nabla u|^2\,d\mu
 +\frac1{n-1}\int_M \phi^2|\nabla |\nabla u||^2\,d\mu
 \notag\\
 \leq& \int_M |\nabla\phi|^2|\nabla u|^2\,d\mu +\int_{\partial M}
 \left(|\nabla u|\partial_\nu |\nabla u|-A^{\p X}(\eta,\eta)|\nabla u|^2\right)\phi^2\,d\sigma .
 \label{eq:wu-arbitrary-dim}
\end{align}
Here $A_M$ is the second fundamental form of $M\subset X$, $\nu$ is the
outward conormal of $\partial M\subset M$, and $\eta$ is a unit normal of
$M\subset X$.

If the boundary satisfies the two-convexity condition $A^{\p X}_2\geq0$, then the
boundary term in \eqref{eq:wu-arbitrary-dim} is nonpositive.

\end{proposition}
\begin{proof}
The proof extends the result of \cite{wuyujiennsc} with \(\Ric^X_2\) replaced with \(\Ric^X_{n-1}\), we summarize the proof below. Denote \(w=|\nabla u|\), plug $\phi\omega$ into the free-boundary stability
inequality and, by  integration by parts, we have
\begin{align}
 \int_M
 \bigl(|A_M|^2+\Ric_X(\eta,\eta)\bigr)\phi^2w^2\,d\mu \notag
 &\leq
 \int_M|\nabla\phi|^2w^2\,d\mu
 -\int_M\phi^2w\Delta w\,d\mu \notag\\
 &+
 \int_{\partial M}
 \bigl(w\partial_\nu w-A^{\partial X}(\eta,\eta)w^2\bigr)
 \phi^2\,d\sigma.
 \label{eq:wu-after-phi-w}
\end{align}
The Bochner formula and the improved Kato inequality for a harmonic function
gives
\[
 w\Delta w\geq \Ric_M(\nabla u,\nabla u)
       +\frac1{n-1}|\nabla w|^2.
\]
Choose an orthonormal basis so that \(e_1=\frac{\nabla u}{|\nabla u|}\) when \(\nabla u \neq 0\). Writing $h_{ij}=A_M(e_i,e_j)$, the Gauss equation, the minimality of $M$ and the assumption \(\Ric_{n-1}^X\geq 0\)
give
\begin{align}
        \Ric_M(e_1,e_1)
        &=\sum_{i=2}^n\operatorname{Rm}_X(e_1,e_i,e_i,e_1)
          -\sum_{j=1}^n h_{1j}^2 \notag\\
        &\geq-\sum_{j=1}^n h_{1j}^2.
        \label{eq:gauss-in-gradient-direction}
\end{align}
Moreover, $H=0$ condition gives (see equation (4.6) in \cite{chodosh-li-stryker}),
\begin{equation}\label{eq:tracefree-row-estimate}
        \sum_{j=1}^n h_{1j}^2
        \leq\frac{n-1}{n}|A_M|^2.
\end{equation}
Consequently
\[
 w\Delta w\geq-\frac{n-1}{n}|A_M|^2w^2
       +\frac1{n-1}|\nabla w|^2.
\]
Substituting it into the stability inequality  and noting that \(\Ric_X(\eta,\eta)\geq 0\) due to \(\Ric^X_{n-1}\geq 0\) gives
\eqref{eq:wu-arbitrary-dim}.

Where
$w\neq0$, the Neumann condition implies that
$e=\nabla u/w$ is tangent to $\partial M$, and
\[
 w\partial_\nu w-A(\eta,\eta)w^2
 =-\big(A(e,e)+A(\eta,\eta)\big)w^2\leq0.
\]
Thus
\begin{equation}\label{eq:wu-arbitrary-dim-no-boundary}
 \frac1n\int_M \phi^2|A_M|^2w^2\,d\mu
 +\frac1{n-1}\int_M \phi^2|\nabla w|^2\,d\mu
 \leq\int_M |\nabla\phi|^2w^2\,d\mu.
\end{equation}
\end{proof}

Then the proof of Theorem \ref{one nonparabolic end} can be carried exactly as in \cite{chodosh-li-stryker} and \cite{wuyujie}, the idea is that existence of two (or more) nonparabolic ends produces a non-constant bounded harmonic function $u$ on $M$ with finite Dirichlet energy. Apply Proposition \ref{thm:wu-arbitrary-dim} with the standard cutoff function $\phi$, we conclude that $|\nabla|\nabla u||=0$ on $M$. Thus $|\nabla u|$ is constant on $M$. Since $u$ has finite Dirichlet energy and volume of $M$ is infinite, we have $|\nabla u|=0$ on $M$, i.e., $u$ is constant. This leads to a contradiction.

Moreover,  a compact boundary component and a nonparabolic end can produce a nontrivial bounded harmonic function with finite Dirichlet energy. So combining Proposition \ref{thm:wu-arbitrary-dim} with the same argument for proof of Theorem \ref{one nonparabolic end} we show that any compact boundary component cannot coexist with a nonparabolic (interior or boundary) end. 

\begin{corollary}[{\cite[Theorem 5.8]{wuyujiennsc}}]\label{thm:no-compact-boundary}
Let $(X^{n+1},\partial X)$ satisfy
$\Ric_{n-1}^X\geq0$ and $A_2^{\partial X}\geq0$, and let
$(M^n,\partial M)\hookrightarrow(X,\partial X)$ be complete, two-sided,
stable, free-boundary minimal, and of infinite volume.  If $M$ has a
nonparabolic end, then every connected component of $\partial M$ is
noncompact.
\end{corollary}

\begin{corollary}
[{\cite[Lemma 5.9]{wuyujiennsc}}]\label{lem:boundary-meets-end}
Under the hypotheses of
Corollary \ref{thm:no-compact-boundary}, let $C_1\subset C_2\subset\cdots$ be a compact exhaustion of M, and let $E_k$ be
nested nonparabolic components of $M\setminus C_k$. Then every connected component $\Gamma$ of
$\partial M$ satisfies
\begin{equation*}\label{eq:boundary-meets-end}
       (\sharp) \quad \quad \quad  (\Gamma\setminus C_k)\cap E_k\neq\varnothing
        \qquad\text{for every }k.
\end{equation*}
\end{corollary}

Corollary \ref{lem:boundary-meets-end} says that every noncompact boundary component meets a nonparabolic end. In fact, if not, a noncompact boundary component and a nonparabolic end that has empty intersection with this boundary component can produce a nontrivial bounded harmonic function with finite Dirichlet energy. Then the rest of the proof is the same as the argument for Theorem \ref{one nonparabolic end}.

\vskip.2cm

\subsection{Infinite volume}
In the paper of Cao-Shen-Zhu \cite{Cao-Shen-Zhu-infinitevolume}, non-collapsing result of complete non-compact minimal hypersurfaces inside the Euclidean space is used. The analogous non-collapsing result for extrinsic balls centered at a boundary point of a free boundary minimal immersion into a compact Riemannian manifold is well-known, see for example \cite{Zhou-Curvature-estimates-for-stable-fbms}. Below we provide a proof for intrinsic balls in the case the ambient manifold is a compact smooth domain in \(\R^{n+1}\).

\begin{theorem}\label{thm:infinite-volume}
Let $\Omega\subset\mathbb R^{n+1}$ be a smooth bounded domain.  Let
\[
        F:(M^n,\partial M)\longrightarrow
        (\overline{\Omega},\partial\Omega)
\]
be a complete noncompact free-boundary minimal immersion.  Then each end of $M$ has infinite volume. In particular,
\[
        \operatorname{Vol}(M)=\infty.
\]
\end{theorem}

\begin{proof}
The non-collapsing volume estimate for interior intrinsic balls is well known. For intrinsic balls centered at a boundary point, we follow the proof idea of \cite{Zhou-Curvature-estimates-for-stable-fbms}. Denote the induced metric on \(M\) as \(g\).
Instead of using radial vector field around an interior point, we adapt the vector field up to second order so that it's tangential along \(\p \Omega\). 
First since \(\Omega\) is bounded, take a uniform collar neighborhood \(U\) of \(\p \Omega
\) and denote the nearest point projection as \(\pi:U \rightarrow \p \Omega\). We denote the outward pointing unit normal along \(\p \Omega\) to be \(\nu\).
There is \(r_0>0\) so that \(B_r(x)\subset U\) for any \(x\in \p \Omega\) and \(r<r_0\).
Then for any \(x\in \p\Omega\) for \(y\in B_{r}(x) \subset  U\) for \(r<r_0\),
\[
X_x(y)=y-x-P_{{\pi(y)}}, \quad P_{\pi(y)}:=\langle \pi(y)-x, \nu_{\pi(y)} \rangle \nu_{\pi(y)}
\]
is the desired ambient vector field.

Indeed we have \(X_x(x)=0\) and 
\[
DX_x(y)-\operatorname{Id}=-D\langle \pi(y)-x,\nu_{\pi(y)} \rangle \nu_{\pi(y)}-\langle \pi(y)-x,\nu_{\pi(y)} \rangle D\nu_{\pi(y)},\]
where the second term vanishes at \(y=x\); and note if \(v\in T_x\p \Omega\), then \(D_v\pi(x) \cdot \nu_x=v\cdot\nu_x=0\).
By further shrinking \(r_0\) so that normal coordinates exists in \(B_r(x)\) for any \(r<r_0\) and \(x\in \p\Omega\), we have there is \(C_0\) such that for any \(y\in B_r(x)\),
\[
\|DX_x(y)-\operatorname{Id}\| \leq C_0|y-x|, \quad |X_x(y)-(y-x)| \leq C_0|y-x|^2.
\]
For any \(p\in \p M\), let \(x=F(p)\) and take the intrinsic ball \(B^M_r(p)\) and denote \(y=F(q)\) for \(q\in B^M_r(p)\), then since \(F\) is minimal, given an orthonormal basis \(e_i \in T_qM\), then for the pullback vector field \(X(p):=X_x(F(p))\) we have,
\begin{align*}
    \dive_{M}X^T&=\dive_{M} (X^{T}+X^{\perp})=\dive_M(X)=\sum_{i=1}^n\nabla_{e_i}X \cdot e_i \\
    &\geq n-nC_0|y-x|\geq n-nC_0d_M(p,q)
\end{align*}
Furthermore, we have by comparison up to second order, for any \(q\in B_{r}^M(p)\),
\[
|X(q)^T|_g\leq |X(q)|\leq |y-x|(1+C_0|y-x|)\leq d_M(p,q)(1+C_0 d_M(p,q)).
\]
We can now apply divergence theorem on \(B_r^M(p)\) to obtain,
\begin{align*}
   n(1-C_0d_{M}(p,q)) \cdot \Vol(B^M_r(p))&\leq \int_{B^M_r(p)} \dive_M X^T\\
   &=\int_{\p B^M_r(p)\cap M^\circ} X^T\cdot \p r+\int_{B^M_r(p)\cap \p M} X^T\cdot \nu_{\p X}\\
   &=\int_{\p B^M_r(p)} X^T\cdot \p r\\
   &\leq r(1+C_0r)\Vol(\p B^M_r(p)),
\end{align*}
where in the last equality we used the free boundary condition.

For almost every \(r<r_0\), \(\frac{d}{dr}\Vol(B^M_r(p))=\Vol(\p B^M_{r}(p))\),
we now denote the absolute continuous function \(V(r)=\Vol(B^M_r(p))\), then for \(r_0\) small enough we have there is \(c_0,\delta>0\) such that for almost every \(r<r_0\),
\[
\frac{d}{dr}\log (V(r))\geq \frac{n}{r}-c_0>\delta>0.
\]
This implies \(\frac{d}{dr} \log(e^{c_0r}\frac{V(r)}{r^n})\geq 0\). We know that as \(r\rightarrow 0\), \(\log(e^{c_0r}\frac{V(r)}{r^n}) \rightarrow \log \frac{w_n}{2}\) for \(w_n\) the volume of unit ball in \(\R^n\).
Integrating over \(r\) we have \(V(r)\geq \frac{w_n}{2e^{c_0r_0}}r^n\) for all \(r<r_0\).

For points a fixed intrinsic distance away from $\partial M$, the
intrinsic-ball estimate of Frensel \cite[Lemma 1]{Cao-Shen-Zhu-infinitevolume},
 gives constants $r>0$ and
$c_{\mathrm{int}}>0$ such that
\[
        \operatorname{Vol}\bigl(B_M(p,r)\bigr)
        \geq c_{\mathrm{int}}r^n
        \qquad\text{if }d_M(p,\partial M)\geq r.
        \tag{2.8}
\]
Above calculation gives
\[
        \operatorname{Vol}\bigl(B_M(q,r)\bigr)
        \geq c_{\partial}r^n
        \qquad\text{for every }q\in\partial M.
        \tag{2.9}
\]

Set $\rho_0=2r$.  If $d_M(p,\partial M)\geq r$, then (2.8) implies
\[
        \operatorname{Vol}\bigl(B_M(p,\rho_0)\bigr)
        \geq c_{\mathrm{int}}r^n.
\]
If $d_M(p,\partial M)<r$, choose $q\in\partial M$ with
$d_M(p,q)<r$.  For every $x\in B_M(q,r)$, the triangle inequality gives
\[
        d_M(p,x)\leq d_M(p,q)+d_M(q,x)<2r=\rho_0.
\]
Hence $B_M(q,r)\subset B_M(p,\rho_0)$, and (2.9) yields
\[
        \operatorname{Vol}\bigl(B_M(p,\rho_0)\bigr)
        \geq c_{\partial}r^n.
\]
We have therefore proved that
\[
        \operatorname{Vol}\bigl(B_M(p,\rho_0)\bigr)\geq c_0>0
        \qquad\text{for every }p\in M.
        \tag{2.10}
\]

It remains to prove that every end has infinite volume. Let $E$ be
an arbitrary end of $M$. Note that $\partial_1E$ is compact. Since
$E$ is unbounded, we can choose inductively a sequence
$\{p_j\}_{j=1}^{\infty}\subset E$ such that
\[
d_M(p_j,\partial_1E)>2\rho_0, \text{ and }
d_M(p_i,p_j)>3\rho_0
\qquad\text{whenever }i\neq j.
\]

Indeed, suppose that $p_1,\ldots,p_j$ have already been chosen.
The set
\[
\bigl\{x\in M:d_M(x,\partial_1E)\leq2\rho_0\bigr\}
\cup
\bigcup_{i=1}^{j}\overline{B_M(p_i,3\rho_0)}
\]
is compact. Since $E$ is unbounded, it is not contained in this
set. We may therefore choose $p_{j+1}\in E$ outside it. This
completes the induction.

For every $j$, the condition \(d_M(p_j,\partial_1E)>2\rho_0\)
implies that
\(
B_M(p_j,\rho_0)\subset E.
\)
Moreover, the balls $B_M(p_j,\rho_0)$ are pairwise disjoint. 
Hence,
\[
\operatorname{Vol}(E)
\geq
\sum_{j=1}^{\infty}
\operatorname{Vol}\bigl(B_M(p_j,\rho_0)\bigr)
\geq
\sum_{j=1}^{\infty}c_0
=
\infty.
\]
As a consequence,
$
\operatorname{Vol}(M)=\infty.
$
\end{proof}

Thus in summary, we have proved the following theorem.
\begin{theorem}\label{OneEnd}
    If \((M^n,\p M) \rightarrow (X^{n+1},\p X)\) is a complete two-sided stable minimal free boundary immersion and \(X\) is a bounded convex domain in \(\R^{n+1}\), then
    \begin{itemize}
        \item \(M\) is compact if and only if \(\p M\) is compact;
        \item if \(M\) is non-compact, then it has only one end, which is a nonparabolic boundary end and satisfies the property \((\sharp)\).
    \end{itemize}
    In particular, this  applies for  $X=\mathbb{B}^{n+1}.$
\end{theorem}

\section{Gromov-Lawson Conformal Metric and Weighted BiRicci Curvature}

\subsection{The conformal stability inequality}
We now pass from the Euclidean ball to the Gulliver--Lawson cylindrical metric. We denote \(|\cdot|_g\) to be the Euclidean norm of a vector in \(\B^{n+1}\),
set
\[
        r(p)=|F(p)|,\qquad t=-\log r\in [0,\infty) ,
\]
then we define the following metric ,
\[
        N:=M\setminus F^{-1}(0), \quad \tilde g=r^{-2}g.
\]
It is not hard to see that $(N,\tilde g)$ is complete. Indeed, since $0<r\leq1$,
we have
$
\tilde g=r^{-2}g\geq g.
$
Hence, if a curve $\gamma:[0,1)\to N$ leaves every compact subset
of $M$, then
\[
L_{\tilde g}(\gamma)\geq L_g(\gamma)=\infty
\]
because $(M,g)$ is complete. It remains to consider a curve $\gamma:[0,1)\to N$ which remains
in a compact subset of $M$ but leaves every compact subset of $N$.
Such a curve must approach $F^{-1}(0)$. Thus, after passing to a
sequence if necessary, we have
\[
r(\gamma(t))\to0
\qquad\text{as }t\to1.
\]
For every $T<1$, we have
\begin{align*}
L_{\tilde g}(\gamma|_{[0,T]})
&=
\int_0^T\frac{|\gamma'(t)|_g}{r(\gamma(t))}\,dt\\
&\geq
\int_0^T
\left|
\frac{d}{dt}\log r(\gamma(t))
\right|\,dt\\
&\geq
\left|
\log r(\gamma(T))-\log r(\gamma(0))
\right|.
\end{align*}
Since $r(\gamma(T))\to0$, the right-hand side tends to infinity.
Therefore $L_{\tilde g}(\gamma)=\infty$.
Thus every divergent curve in $N$ has infinite $\tilde g$-length,
and hence $(N,\tilde g)$ is complete.

The ambient punctured ball with metric $r^{-2}g_{\mathrm{Euc}}$ is the half
cylinder
\[
        (\B^{n+1}\setminus\{0\},r^{-2}g_{\mathrm{Euc}})
        \cong
        (\Sph^n\times[0,\infty),g_{\Sph^n}+dt^2)=:(X,\partial X).
\]
The spherical boundary $\Sph^n=\{r=1\}$ becomes the totally geodesic boundary slice $\{t=0\}$ of \(X\).
Since conformal maps preserve angles, the free-boundary condition is unchanged, i.e.
$N$ meets $\partial X$ orthogonally, and \(\p N\) is totally geodesic in \((N,\tilde{g})\).

Since \(F:(M^n,g) \rightarrow \B^{n+1}\) is a complete smooth immersion, the preimage \(F^{-1}(0)\) cannot have any accumulation point in the induced metric \((M,g)\), therefore \(F^{-1}(0)\) is a discrete set in \((M,g)\). In fact, if $n\leq 4$, we have a better situation since $F$ does not pass through the origin by Corollary \ref{cor:strict-inradius-unit-ball}.


We first compute the stability inequality under conformal change.

\begin{proposition}\label{prop:conf-stab-boundary}
Let $M^n\to\overline{\B^{n+1}}$ be a two-sided stable free-boundary minimal
immersion, and set $\tilde g=r^{-2}g$ on $N=M\setminus F^{-1}(0)$.  Then for all
smooth $\psi$ with compact support in $N$ as a manifold with boundary,
\begin{align}
        \int_N |\tilde\nabla\psi|^2\,d\tilde\mu
        &\geq
        \int_N
        \left(
        r^2|A|^2-\frac{n(n-2)}2
        +\frac{n^2-4}{4}|dr|^2
        \right)\psi^2\,d\tilde\mu  \notag \\
        &\qquad
        + \frac n2\int_{\partial N}\psi^2\,d\tilde\sigma .
        \label{eq:conf-stab-boundary}
\end{align}
\end{proposition}

\begin{proof}
We start from \eqref{eq:fb-stability} and use the usual Gulliver--Lawson conformal weight,
\[
        \varphi=r^{(2-n)/2}\psi.
\]
Since $d\mu=r^n\,d\tilde\mu$, $|\nabla f|^2=r^{-2}|\tilde\nabla f|^2$, and
$r=1$ along $\partial M$, the stability inequality becomes
\begin{align}\label{eq:stabI-1}
    \int_N r^{n-2}|\tilde\nabla\varphi|^2\,d\tilde\mu
        \geq
        \int_N r^n|A|^2\varphi^2\,d\tilde\mu
        +\int_{\partial N}\varphi^2\,d\tilde\sigma .
\end{align}
First,
\[
        \tilde\nabla\varphi
        =
        r^{(2-n)/2}
        \left(
        \tilde\nabla\psi-\frac{n-2}{2}\psi\,\tilde\nabla\log r
        \right).
\]
Thus expanding $\tilde\nabla\varphi$ gives
\begin{align*}
        r^{n-2}|\tilde\nabla\varphi|^2
        &=
        |\tilde\nabla\psi|^2
        +\frac{(n-2)^2}{4}
        |\tilde\nabla\log r|_{\tilde g}^2\psi^2  
        -\frac{n-2}{2} \langle\tilde\nabla(\psi^2),\tilde\nabla\log r\rangle .
\end{align*}
The minimality of $M$ and the
standard conformal Laplacian formula gives,
\[
        \tilde\Delta\log r=n-n|dr|^2 .
\]
Thus the cross term satisfies
\[
        -\frac{n-2}{2}\int_N
        \langle\tilde\nabla(\psi^2),\tilde\nabla\log r\rangle\,d\tilde\mu
        =
        \frac{n-2}{2}\int_N
        (n-n|dr|^2)\psi^2\,d\tilde\mu
        -\frac{n-2}{2}\int_{\partial N}\psi^2\,d\tilde\sigma ,
\]
Here we used
$|\tilde\nabla\log r|_{\tilde g}^2=|dr|_g^2$, 
and the outward $\tilde g$-unit conormal is again $\eta$ on $\partial N$
and $\partial_\eta \log r=1$.

Now plug these computations into (\ref{eq:stabI-1}) and we obtained 
\eqref{eq:conf-stab-boundary}.  
\end{proof}

\subsection{Spectral Weighted BiRicci Curvature}
For an ordered orthonormal pair $(e_1,e_2) \in T_pM$ at a point \(p\) in a Riemannian $n$-manifold \(N\), set
\begin{equation}\label{eq:alpha-biric-definition}
 \biRic_\alpha(e_1,e_2)
 =
 \sum_{i=2}^n R(e_1,e_i,e_i,e_1)
 +\alpha\sum_{j=3}^n R(e_2,e_j,e_j,e_2),
\end{equation}
and let
\[
        \Lambda_\alpha(p)
        =
        \min_{(e_1,e_2)}\biRic_\alpha(e_1,e_2),
\]
where the minimum is taken among all orthonormal pairs in \(T_p N\). When \(\alpha=1\), this is the biRicci curvature introduced in \cite{shenyingyerugang}.
Note that \(\Lambda_{\alpha}\) is a Lipschitz function on \(N\).

We have the following spectral $\alpha$-bi-Ricci inequalities with Robin Boundary condition.
\begin{proposition}\label{prop:spectral-biric-boundary}
Let \((M^4,\p M)\rightarrow (\B^5,\Sph^4)\) be a complete two-sided stable  free boundary minimal hypersurface, and  let \((N^4,\tilde{g})\) be the Gromov--Lawson conformal metric applied to \(M\setminus F^{-1}(0)\).
There is a smooth function $V$ on $N$ satisfying
\[
        V\geq 1-\tilde{\Lambda}_1
\]
such that
\begin{equation}\label{eq:biric-robin}
        \int_N |\tilde\nabla\psi|^2\,d\tilde\mu
        \geq
        \int_N V\psi^2\,d\tilde\mu
        +2\int_{\partial N}\psi^2\,d\tilde\sigma
\end{equation}
for every smooth $\psi$ with compact support in $N$ as a manifold with
boundary.
\end{proposition}

\begin{proof}
We note that by Proposition \ref{prop:conf-stab-boundary}, it suffices to prove the required pointwise lower bound for $V=r^2|A|^2-\frac{n(n-2)}{2}+\frac{n^2-4}{4}|dr|^2$ when $n=4$.  This is a local computation
identical to the closed case in \cite{Chodosh-Li-Minter-Stryker}. We directly take the following inequality which was proved in \cite[Proposition 3.1]{Chodosh-Li-Minter-Stryker}
\[
        V
        \geq 1-\tilde\Lambda_{1}.
\]
 Applying Proposition
\ref{prop:conf-stab-boundary} with $n=4$ gives
\[
        \int_N |\tilde\nabla\psi|^2\,d\tilde\mu
        \geq
        \int_N V\psi^2\,d\tilde\mu
        +2\int_{\partial N}\psi^2\,d\tilde\sigma
\]
for every smooth $\psi$ with compact support in $N$ as a manifold with
boundary. 
\end{proof}

\begin{proposition}
\label{prop:weighted-biric-dim-five}
Let \((M^5,\p M)\rightarrow (\B^6,\Sph^5)\) be a complete two-sided stable  free boundary minimal hypersurface, and let \((N^5,\tilde{g})\) be the Gromov--Lawson conformal metric applied to \(M\setminus F^{-1}(0)\).  For
\begin{equation}\label{eq:mazet-parameters}
        a=\frac{11}{10},
        \qquad
        \alpha=\frac{40}{43},
        \qquad
        \delta=\frac3{10},
\end{equation}
there is a smooth function $V$ on $N$ such that
\begin{equation}\label{eq:weighted-biric-lower-bound}
        V\geq\delta-\widetilde\Lambda_\alpha
\end{equation}
and
\begin{equation}\label{eq:weighted-biric-robin}
 \int_N|\tilde\nabla\varphi|^2\,d\tilde\mu
 \geq
 \frac1a\int_N V\varphi^2\,d\tilde\mu
 +\frac52\int_{\partial N}\varphi^2\,d\tilde\sigma
\end{equation}
for every smooth compactly supported $\varphi$ on $N$ as a manifold with
boundary.
\end{proposition}

\begin{proof}
As in the previous Proposition, this result also follows from Proposition \ref{prop:conf-stab-boundary} if we estimate $V=a(r^2|A|^2-\frac{n(n-2)}{2}+\frac{n^2-4}{4}|dr|^2)$ when $n=5$. This was done in
\cite[Section 3, proof of Theorem 3.1]{mazet}, that is
\[V\geq \delta-\tilde{\Lambda}_\alpha.\]
Finally, Proposition \ref{prop:conf-stab-boundary} gives
\[
 \int_N|\tilde\nabla\varphi|^2
 \geq
 \frac1a\int_NV\varphi^2+\frac52\int_{\partial N}\varphi^2.
\]
\end{proof}

\section{Volume Estimates of Free Boundary \(\mu\)-Bubbles}
\subsection{Free boundary \(\mu\)-bubbles}\label{FB-Bubble}
We say \((Y^n,\p Y,\tilde{g})\) is a Riemannian band with boundary
if it is a smooth compact manifold with corner, whose boundary \(\p Y\) is decomposed as
\[
        \partial Y=\partial_0Y\cup\partial_-Y\cup\partial_+Y, \quad \p_- Y \cap \p_+ Y =\emptyset
\]
where each \(\p_i Y, i\in \{0,+,-\}\) is a smooth \(n-1\) manifold with boundary, and \(\p_{\pm} Y\) meets with \(\p_0 Y\) at an angle $\theta(x)\in (0,\frac{\pi}{2}]$.

Let $h$ be a smooth function on 
$ Y\setminus(\partial_-Y\cup\partial_+Y)$ such that
\[
        h\to+\infty\text{ on }\partial_+Y,
        \qquad
        h\to-\infty\text{ on }\partial_-Y.
\]
Choose a regular value $c_0$ of $h$ and set
\[
        \Omega_0=\{h>c_0\}\supset \partial_+ Y.
\]
Given a smooth positive function \(u\) defined on $Y$, we define the following warped \(\mu\)-bubble functional for Caccioppoli sets $\Omega\subset Y$ such that \(\Omega\triangle\Omega_0 \Subset Y \setminus(\p_-Y\cup \p_+Y)\),
\begin{equation}\label{eq:relative-functional}
        \cA(\Omega)
        =
        \int_{\partial^*\Omega\cap\Int Y}u\,d\cH^{n-1}_{\tilde g}
        -
        \int_Y(\chi_\Omega-\chi_{\Omega_0})hu\,d\tilde\mu .
\end{equation}
Since each Caccioppoli set \(\Omega\) containing a neighborhood of \(\partial_+Y\), \(\cA(\Omega)\) is finite, as we used renormalization by \(\Omega_0\).

We recall the existence and regularity of free boundary \(\mu\)-bubble studied in \cite{chodosh-li-bubble, wuyujie}, and we recall the variation formulas.

\begin{lemma}[{ \cite[Lemma 6.2]{wuyujie}}]\label{lem:relative-mububble}
There is a minimizer $\Omega$ to the functional \eqref{eq:relative-functional} realizing,
\[
A_0:=\inf_{\Omega:\text{ Caccioppoli set, }\Omega\triangle\Omega_0 \Subset Y \setminus(\p_-Y\cup \p_+Y)} \cA(\Omega)\in \R.
\]
Moreover
$\Omega\triangle\Omega_0$ is compactly contained in $Y^\circ\cup\partial_0Y$,
and the reduced boundary
\[
        \Sigma:=\partial^*\Omega\cap (Y^\circ\cup\partial_0Y)
\]
is smooth when \(n\leq 7\); $\Sigma$ meets $\partial_0 Y$ orthogonally along $\partial\Sigma$.
\end{lemma}

\begin{theorem}[{\cite[Theorem 6.3]{wuyujie}}]\label{thm:relative-variation}
Let $\Omega$ be a smooth minimizer from Lemma \ref{lem:relative-mububble},
$\nu$ be the outward unit normal of \(\Sigma\) inside $\Omega$.  Then
\begin{equation}\label{eq:relative-first-variation}
H_{\Sigma}=h-\p_{\nu}\log u,
        \qquad\text{on }\Sigma.
\end{equation}

For every admissible normal variation with speed $\phi$, the second variation
satisfies,
\begin{align}
0\leq\delta^2\cA(\phi)
&=
\int_\Sigma u|\nabla^\Sigma\phi|^2
        -u\left(|A_\Sigma|^2+\widetilde{\Ric}(\nu,\nu)\right)\phi^2
        \,d\mu_\Sigma
        \notag\\
&\quad
 +\int_\Sigma
        \phi^2\left(\tilde\Delta u-\Delta^\Sigma u-\partial_\nu(hu)\right)
        \,d\mu_\Sigma
        \notag\\
&\quad
-\int_{\partial\Sigma}
        u\phi^2\,\mathrm{II}_{\partial_0 Y}(\nu,\nu)
        \,d\sigma_{\partial\Sigma}.
\label{eq:relative-second-variation}
\end{align}
Moreover,
\[
        |A_\Sigma|^2+\widetilde{\Ric}(\nu,\nu)
        =
        \frac12\left(\widetilde R-R_\Sigma+|A_\Sigma|^2+H_\Sigma^2\right).
\]
\end{theorem}
We note that from now on, \(\Sigma\) will be a \(\mu\)-bubble obtained on some Riemannian band with boundary in \((N,\tilde{g})\): obtained from applying the Gromov--Lawson conformal metric to \(M\setminus F^{-1}(0)\); and we always consider the induced metric on \(\Sigma\) from \(N\) and denote the induced Levi-Civita connection as \(\nabla^{\Sigma}\).

\subsection{Spectral Ricci Positivity of \(\mu\)-Bubble}
By Propositions \ref{prop:spectral-biric-boundary} and \ref{prop:weighted-biric-dim-five},
there is a positive function $w$ on \(N\) satisfying
\begin{equation}\label{eq:weighted-positive-solution}
        -a\frac{\tilde\Delta w}{w}
        =V\geq\delta-\widetilde\Lambda_\alpha,
        \qquad
        \partial_\eta\log w=\frac{n}{2},
\end{equation}
with \(a=\alpha=\delta=1\) when \(n=4\) and \(\alpha, a, \delta\) are given in (\eqref{eq:mazet-parameters}) when \(n=5\).

We first prepare a perturbation lemma (see \cite[Lemma 4.3]{chen-hong-nonexistence}) to obtain dihedral angles (no more than $\pi/2$) which is required in the definition of Riemannian band with boundary in this section. This later allows us to prove the existence of free boundary $\mu$-bubbles.
\begin{lemma}\label{perturbation}
Let \((V,\p V,\tilde{g})\) be a smooth compact manifold with corner, whose boundary \(\p V\) is decomposed as
\[
        \partial V=\partial_0V\cup\partial_-V\cup\partial_+V \quad \p_- V \cap \p_+ V =\emptyset
\]
where each \(\p_i V, i\in \{0,+,-\}\) is a smooth \(n-1\) manifold with boundary, and \(\p_{\pm} V\) meets with \(\p_0 V\) at an angle $\theta(x)\in (0,\pi)$. 

Let $\psi$ be a smooth function on $(V,\partial V, \tilde{g})$. Suppose
\[
       \psi|_{\partial_{\pm} V}=\pm A , \ \ \ |\psi|\leq A.
\]
Denote
$q=\operatorname{Lip}(\psi)$.  Given $\sigma>0$ and arbitrarily small
neighborhoods $U_\pm$ of $\partial_\pm V$, there is a smooth function
$\widetilde\psi$ on $V$ such that
\begin{enumerate}[label=\textnormal{(\roman*)}]
\item $\widetilde\psi=\psi$ outside $U_+\cup U_-$;
\item the two faces $\{\widetilde\psi=\pm A\}$ meet $\partial_0V$ at
non-obtuse interior dihedral angles;
\item $\operatorname{Lip}(\widetilde\psi)\leq q+\sigma$.
\end{enumerate}
If one of the original faces $\partial_\pm V$ is already non-obtuse, that face is left
unchanged.  The displacement of each modified face can be made arbitrarily
close to $\partial_\pm V$.
\end{lemma}

\begin{proof}
We give the construction at the upper face; the lower face is treated by
applying the same argument to $-\psi$.  Denote
\[ 
        \Gamma_+=\partial_+ V \cap\partial_0V.
\]
If the interior angle between $\partial_+ V$ and $\partial_0 V$ is already at most
$\pi/2$ everywhere, no modification is made.  Otherwise choose $\tau>0$ so
small that $A-\tau$ is a regular value of both $\psi$ and
$\psi|_{\partial_0 V}$ and that the region between the levels $A-\tau$ and $A$
is contained in $U_+$.  Since $\psi$ is $q$-Lipschitz,
\begin{equation}\label{eq:common-level-distance}
 \dist_{\tilde g}\bigl(\partial_+ V,\psi^{-1}(A-\tau)\bigr)
 \geq\frac{\tau}{q}.
\end{equation}

Fix $0<\tau_1\ll\tau/q$.  In a $\tau_1$-tubular neighborhood of $\partial_+ V$,
choose a smooth vector field $X$ with $|X|\leq 1$ which is transverse to $\partial_+ V$, points toward
$V$, and is tangent to $\partial_0 V$ along $\Gamma_+$.  In particular, take $X$ to be the unit normal to $\Gamma_+$ in $\partial_0 V$. Let $\Phi_t$ be its local flow of $X$.

Use collar coordinates $(z,s)\in\Gamma_+\times[0,s_0)$ on $\partial_+ V$, where
$\xi=\partial_s$ is the unit inward conormal of $\Gamma_+$ in $\partial_+ V$.  Choose
$f:\partial_+ V\to[0,1]$ with $f(z,0)=0$ and define the flow graph
\[
        G_f=\{\Phi_{\tau_1f(y)}(y):y\in \partial_+ V\}.
\]
Near the edge it is parametrized by
\[
        F(z,s)=\Phi_{\tau_1f(z,s)}(z,s).
\]
Because $\partial_t\Phi_t=X\circ\Phi_t$, the chain rule gives at $s=0$
\[
 F_*(\partial_s)
 =\partial_s+\tau_1(\partial_s f)X.
\]
Moreover, $f(z,0)\equiv0$ implies
$F_*(\partial_{z_i})=\partial_{z_i}$.  Consequently, at
$p\in\Gamma_+$,
\begin{align*}
 T_pG_f
 &=T_p\Gamma_+\oplus
   \operatorname{span}\{\xi+\lambda X\},\\
 T_p\partial_0 Y
 &=T_p\Gamma_+\oplus\operatorname{span}\{X\},
 \qquad
 \lambda=\tau_1\partial_s f(p).
\end{align*}

Set $c(p)=\langle\xi,X\rangle \in (-1,1)$ by transversality, so \(1+2c\lambda+\lambda^2=(c+\lambda)^2+1-c^2>0\).  With the inward orientations just chosen,
the original interior angle $\theta_0$ satisfies
$\cos\theta_0=c(p)$.  The interior angle $\theta_f$ of the flow graph is
therefore determined by
\begin{equation}\label{eq:common-flow-graph-angle}
 \cos\theta_f
 =\frac{\langle\xi+\lambda X,X\rangle}{|\xi+\lambda X|}
 =\frac{c+\lambda}{\sqrt{1+2c\lambda+\lambda^2}}.
\end{equation}
Thus no change is needed where $c\geq0$.  Where $c<0$, any choice of 
$\lambda(z)\geq-c$ makes $\theta_f\leq\pi/2$; indeed,
$\cos\theta_f\to1$ as $\lambda(z)\to+\infty$.  Compactness of $\Gamma_+$ allows
$\lambda$ to be chosen smoothly and uniformly large enough.  This is
compatible with $0\leq f\leq1$: if
$\chi:[0,\infty)\to[0,1]$ is smooth and $\chi(r)=r$ near zero, set
\[
        f(z,s)=\chi\left(\frac{\lambda(z)}{\tau_1}s\right)
\]
near the edge.  Then
$\tau_1\partial_s f(z,0)=\lambda(z)$, while the transition occurs on a scale
of order $\tau_1/\lambda(z)$.

The graph $G_f$ lies in the $\tau_1$-neighborhood of $\partial_+ V$.  By
\eqref{eq:common-level-distance}, it is disjoint from
$\psi^{-1}(A-\tau)$ when $\tau_1<\tau/q$.  Define $\psi_+$ to equal $\psi$
on $\{\psi\leq A-\tau\}$, to equal $A$ on $G_f$ and beyond it, and interpolate
smoothly between the two faces.  Their boundary values differ by $\tau$ and
their distance is at least $\tau/q-\tau_1$, so the interpolation and an
arbitrarily small smoothing may be chosen with
\begin{equation}\label{eq:common-interpolation-lipschitz}
 \operatorname{Lip}(\psi_+)
 \leq\frac{q\tau}{\tau-q\tau_1}+o(1).
\end{equation}
The right-hand side tends to $q$ as $\tau_1\downarrow0$.  Hence it is at most
$q+\sigma/2$ after choosing $\tau_1$ and the smoothing error sufficiently
small.  Applying the same construction to $-\psi_+$ near the lower face,
with Lipschitz increase at most $\sigma/2$, gives $\widetilde\psi$.  The two
supports are disjoint, and all three conclusions follow.
\end{proof}

\begin{theorem}
\label{thm:weighted-relative-bubble-dim-five}
Let \(V \subset (N^n,\tilde{g})\) be a smooth manifold with corner as defined in Lemma \ref{perturbation} with \(\p_0 V \subset \p N\) and  \((\p_{\pm} V \setminus (\p_0 V))\subset N^{\circ}\). 

Assume \(d_{\tilde{g}}(\p_- V, \p_+V)\geq  L_0\), where \(L_0=4\pi+1\) when \(n=4\), and \(L_0=100\pi\) when \(n=5\), then there is a smooth prescribed mean curvature function \(h\in C^{\infty}(Y)\) for some subset \(Y \subset V\), and \(Y\) is a Riemannian band with boundary  such that the following weighted \(\mu\)-bubble functional defined for \(\Omega\triangle\Omega_0 \Subset Y\setminus(\p_-Y \cup \p_+Y)\)
\begin{equation}\label{eq:relative-weighted-functional-dim-five}
 \cA_a(\Omega)
 =
 \int_{\partial^*\Omega\cap\Int Y}w^a\,d\cH^{n-1}_{\tilde g}
 -
 \int_Y(\chi_\Omega-\chi_{\Omega_0})hw^a\,d\tilde\mu.
\end{equation}
has a minimizer \(\Omega\) whose smooth boundary \(\Sigma\) lies in the closed $L_0$-neighborhood of $\partial_+Y$, and satisfies
\begin{equation}\label{eq:weighted-bubble-spectral-ricci-dim-five}
 \frac4{4-a}\int_\Sigma|\tilde{\nabla}^\Sigma\psi|^2\,d\tilde{\mu}_\Sigma
 \geq
 \int_\Sigma
 \left(\frac\delta2-\alpha\lambda_{\Ric}^\Sigma\right)
 \psi^2\,d\tilde{\mu}_\Sigma
 +\frac{na}{2}\int_{\partial\Sigma}\psi^2\,d\tilde{\sigma}_{\partial\Sigma}
\end{equation}
for every smooth function $\psi$ on $\Sigma$.
\end{theorem}
\begin{proof}
We first use Lemma \ref{perturbation} above to
construct a Riemannian band with boundary where the barriers form non-obtuse interior dihedral angles.

Let
\[
        d=\dist_{\tilde g}(\,\cdot\,,\partial_+V).
\]
In the case $n=4$, choose $\varepsilon>0$ so small that
\[
        4\pi+4\varepsilon
        <
        \dist_{\tilde g}(\partial_-V,\partial_+V),
\]
and choose a smooth approximation $\rho_0$ of $d$ satisfying
\[
 |\tilde\nabla\rho_0|<2,
 \qquad
 |\rho_0-d|<\varepsilon,
 \qquad
 \rho_0=0\quad\hbox{on }\partial_+V.
\]
After arbitrarily small changes, assume that $\varepsilon$ and
$4\pi+2\varepsilon$ are regular values of both $\rho_0$ and
$\rho_0|_{\partial N}$.  Set
\[
        \varphi_4
        =
        \frac{\rho_0-\varepsilon}{4+\varepsilon/\pi}
        -\frac\pi2.
\]
Then
\[
        q_4:=\operatorname{Lip}(\varphi_4)
        \leq\frac{2}{4+\varepsilon/\pi}<\frac12.
\]
Apply Lemma \ref{perturbation} to \(\varphi_4\) with $A=\pi/2$ and
$0<\sigma_4<1/2-q_4$.  Denote the resulting function by
$\widetilde\varphi_4$ and define
\[
        h=-\tan\widetilde\varphi_4
        \quad\hbox{on}\quad
        \left\{-\frac\pi2<\widetilde\varphi_4<\frac\pi2\right\}=:Y.
\]
The new barriers are formed by \(\p_{\pm}Y:=\{\widetilde{\varphi}_4=\pm \frac{\pi}{2}\}\) and
\[
 |\tilde\nabla h|
 =(1+h^2)|\tilde\nabla\widetilde\varphi_4|
 <\frac12(1+h^2).
\]
Consequently,
\begin{equation}\label{eq:unified-four-barrier}
        1+h^2-2|\tilde\nabla h|\geq0.
\end{equation}
The adjustment may be supported in an $\varepsilon$-neighborhood of the
original levels, so \(Y\) is contained in
$B_{4\pi+4\varepsilon}^{\tilde g}(\partial_-V)$.

In the case $n=5$, Denote
\[
        T_5=11\pi\sqrt{\frac{88}{15}}.
\]
Choose $\varepsilon>0$ small enough that
$
        2\bigl(\varepsilon+(1+\varepsilon)T_5\bigr)
        +\varepsilon<100\pi,
$
and take a smooth approximation $\rho_0$ of $d$ such that
$
        \frac12d\leq\rho_0\leq2d,
     $ and $
        |\tilde\nabla\rho_0|\leq2.
$
We may arrange that
\[
        s_-:=\varepsilon,
        \qquad
        s_+:=\varepsilon+(1+\varepsilon)T_5
\]
are regular values of both $\rho_0$ and $\rho_0|_{\partial N}$.  Define
\[
        \psi_5=\frac{\rho_0-\varepsilon}{1+\varepsilon}-\frac{T_5}{2}.
\]
Note
\[
        q_5:=\operatorname{Lip}(\psi_5)
        \leq\frac{2}{1+\varepsilon}<2.
\]
Apply Lemma \ref{perturbation} with $A=T_5/2$ and
$0<\sigma_5<2-q_5$.  Denote $\widetilde\psi_5$ as the adjusted function, set
\[
        h=-k(\widetilde\psi_5+\frac{T_5}{2}),
\]
where
\[
 k(t)=
 \sqrt{\frac{33}{10}}\,
 \tan\left(
 \frac1{11}\sqrt{\frac{15}{88}}\,t-\frac\pi2
 \right),
 \qquad 0<t<T_5.
\]
Again the blow-up faces of $h$ are exactly the adjusted non-obtuse barriers.
Moreover,
\[
        |\tilde\nabla\widetilde\psi_5|<2,
        \qquad
        -k'=\frac3{44}+\frac5{242}k^2.
\]
Since $a=11/10$ and $\delta=3/10$,
\begin{align}
 |a\,dh(\nu)|
 &\leq
 2a\left(\frac3{44}+\frac5{242}h^2\right) \notag\\
 &=\frac\delta2+\frac1{22}h^2.
\label{eq:unified-five-barrier}
\end{align}
The comparison $d\leq2\rho_0$, the choice of $\varepsilon$, and the
arbitrarily small support of the angle adjustment place this band inside
$B_{100\pi}^{\tilde g}(\partial_+V)$.

In either case, choose a regular value $c_0$ of $h$ and Denote
$\Omega_0=\{h>c_0\}$ and minimize the following \(\cA_a (\cdot)\) among Caccioppoli sets such that \(\Omega\triangle \Omega_0\Subset(Y \setminus (\p_+ Y\cup \p_-Y))\),
\[
 \cA_a(\Omega)
 =
 \int_{\partial^*\Omega\cap\Int Y}w^a\,d\cH^{n-1}_{\tilde g}
 -
 \int_Y(\chi_\Omega-\chi_{\Omega_0})hw^a\,d\tilde\mu .
\]
Lemma \ref{lem:relative-mububble} gives a
smooth minimizing boundary which meets $\partial N$ orthogonally.  

Let $\nu$ be the outward unit normal of the minimizing set along $\Sigma$.
The first variation is
\[
        H=h-a\,d\log w(\nu).
\]
Apply Theorem \ref{thm:relative-variation} with weight $w^a$ and set
$\phi=w^{-a/2}\psi$.  The boundary term involving
$\mathrm{II}_{\partial N}(\nu,\nu)$ vanishes because $\partial N$ is totally
geodesic.  Moreover,
\begin{align*}
 \frac{\tilde\Delta(w^a)-\Delta^\Sigma(w^a)}{w^a}
 &={}
 a\frac{\tilde\Delta w-\Delta^\Sigma w}{w}
 +a(a-1)\bigl(d\log w(\nu)\bigr)^2,\\
 \frac{\partial_\nu(hw^a)}{w^a}
 &={}dh(\nu)+ah\,d\log w(\nu).
\end{align*}
Integration by parts in the tangential Laplacian term
gives
\begin{align*}
&\int_\Sigma w^a
 \left(|\nabla^\Sigma\phi|^2
       -aw^{-1}\Delta^\Sigma w\,\phi^2\right)\\
={}&
 \int_\Sigma
 \left(
 |\nabla^\Sigma\psi|^2
 +aw^{-1}\psi\langle\nabla^\Sigma w,\nabla^\Sigma\psi\rangle
 -\left(a-\frac{a^2}{4}\right)
  \psi^2w^{-2}|\nabla^\Sigma w|^2
 \right)  \\
&\quad
 -a\int_{\partial\Sigma}\psi^2\partial_\mu\log w ,
\end{align*}
where $\mu$ is the outward conormal of $\partial\Sigma\subset\Sigma$.  The
estimate
\[
 w^{-1}\psi\langle\nabla^\Sigma w,\nabla^\Sigma\psi\rangle
 \leq
 \frac1{4-a}|\nabla^\Sigma\psi|^2
 +\frac{4-a}{4}\psi^2w^{-2}|\nabla^\Sigma w|^2.
\]
exactly cancels the last interior gradient term.  Using also
$h=H+a\,d\log w(\nu)$, the second variation therefore gives 
\begin{align}
 \frac4{4-a}\int_\Sigma|\nabla^\Sigma\psi|^2
 \geq{}&
 \int_\Sigma
 \left(
 -a\frac{\tilde\Delta w}{w}
 +|A_{\Sigma}|^2+\widetilde{\Ric}(\nu,\nu)
 \right)\psi^2 \notag\\
 &+\int_\Sigma
 \left(
 dh(\nu)+a\bigl(d\log w(\nu)\bigr)^2
 +aH\,d\log w(\nu)
 \right)\psi^2 \notag\\
 &+a\int_{\partial\Sigma}\psi^2\partial_\mu\log w.
\label{eq:unified-common-second-variation}
\end{align}

Up to this point the argument is common to both dimensions $n=4$ and $n=5$.  We now estimate
the remaining interior terms separately.  In the case $n=4$, using
$a=\alpha=\delta=1$ and $H=h-d\log w(\nu)$, the second variation gives
\begin{align*}
 \frac43\int_\Sigma|\nabla^\Sigma\psi|^2
 \geq{}&
 \int_\Sigma
 \left(
 -\frac{\tilde\Delta w}{w}
 +|A_\Sigma|^2+\widetilde{\Ric}(\nu,\nu)-\frac12H^2-\frac12
 \right)\psi^2  \\
&+\frac12\int_\Sigma(1+h^2-2|\tilde\nabla h|)\psi^2
 +\int_{\partial\Sigma}\psi^2\partial_\mu\log w .
\end{align*}
 The lower bound
$-\tilde\Delta w/w\geq1-\tilde\Lambda_1$ and the Gauss equation imply
the pointwise estimate
\[
 -\frac{\tilde\Delta w}{w}
 +|A_\Sigma|^2+\widetilde{\Ric}(\nu,\nu)-\frac12H^2-\frac12
 \geq
 \frac12-\lambda_{\Ric}^{\Sigma},
\]
whose proof follows from the argument from equation (4.4) onwards in \cite{Chodosh-Li-Minter-Stryker}.
The  term \((1+h^2-2|\tilde\nabla h|)\) is nonnegative by construction. Together with the boundary term this gives the desired inequality when \(n=4\).

In the case $n=5$, choose $e_1\in T\Sigma$ with
$\Ric^\Sigma(e_1,e_1)=\lambda_{\Ric}^{\Sigma}$.  The Gauss equation, the lower
bound $-a\tilde\Delta w/w\geq\delta-\widetilde\Lambda_\alpha$, and
$H=h-a\,d\log w(\nu)$ give (see \cite[equation (12)]{mazet} or \cite[Lemma 3.2]{tam2024estimatesstableminimalhypersurfaces})
\[
 \frac4{4-a}\int_\Sigma|\nabla^\Sigma\psi|^2
 \geq
 \int_\Sigma
 \left(\delta-\alpha\lambda_{\Ric}^{\Sigma}
       +K+dh(\nu)\right)\psi^2
 +a\int_{\partial\Sigma}\psi^2\partial_\mu\log w ,
\]
where
\[
 K=
 |A|^2+\alpha HA_{11}
 -\alpha\sum_{j=1}^4A_{1j}^2
 +a\bigl(d\log w(\nu)\bigr)^2
 +aH\,d\log w(\nu).
\]
Decomposing $A$ into trace and
traceless variables, section 4.2 of \cite{mazet} shows  $K\geq h^2/22$.  Together with the estimate (\ref{eq:unified-five-barrier}) and that \(a>1\), we have
\[
        |dh(\nu)|\leq\frac\delta2+\frac1{22}h^2,
\]
we get $\delta/2+K+dh(\nu)\geq0$.

Both cases therefore leave the common interior lower bound
$\delta/2-\alpha\lambda_{\Ric}^{\Sigma}$.
Finally, orthogonality of $\Sigma$ and $\partial N$ implies
$\mu=\eta$ along $\partial\Sigma$. This gives the boundary term and completes the proof of the theorem.
\end{proof}

The inequality (\ref{eq:weighted-bubble-spectral-ricci-dim-five}) implies  that the free boundary \(\mu\)-bubble with totally geodesic boundary satisfies the following spectrally positive Ricci curvature condition.
We can now apply the boundary version of spectral Bishop-Gromov comparison of  Antonelli-Xu \cite{antonelli-xu}, proved in \cite{lijia2026} to obtain volume or diameter estimates.
\begin{equation}\label{eq:bubble-neumann-spectral-ricci}
\begin{cases}
        -\frac{4}{4-a}\Delta^\Sigma w+{\alpha}\lambda_{\Ric}^{\Sigma}w
        \geq\frac{\delta}{2} w
        &\text{on }\Sigma,\\[4pt]
        \partial_\mu w=0
        &\text{on }\partial\Sigma.
\end{cases}
\end{equation}

\begin{corollary}
\label{cor:relative-bubble-estimates}
When \(n=4\), the $3$-dimensional relative bubble $\Sigma$ obtained in Theorem
\ref{thm:weighted-relative-bubble-dim-five} satisfies
\[
        \operatorname{diam}_{\tilde g}(\Sigma)\leq 2\pi,
        \qquad
        \Vol_{\tilde g}(\Sigma)\leq 16\pi^2.
\]
\end{corollary}

\begin{proof}
Dropping the nonnegative boundary term in
(\ref{eq:weighted-bubble-spectral-ricci-dim-five}) gives
\[
        \int_\Sigma
        \left(\frac43|\nabla^\Sigma\psi|^2
        +\lambda_{\Ric}^{\Sigma}\psi^2\right)\,d\mu_\Sigma
        \geq
        \frac12\int_\Sigma\psi^2\,d\mu_\Sigma .
\]
The first Neumann eigenfunction  of the preceding inequality implies
\begin{equation}\label{eq:bubble-neumann-spectral-ricci}
\begin{cases}
        -\dfrac43\Delta^\Sigma w+\lambda_{\Ric}^{\Sigma}w
        \geq\dfrac12w
        &\text{on }\Sigma,\\[4pt]
        \partial_\mu w=0
        &\text{on }\partial\Sigma.
\end{cases}
\end{equation}
Moreover, $\partial\Sigma$ is totally geodesic in $\Sigma$.  Indeed, if
$X,Y\in T(\partial\Sigma)$, then the orthogonality of $\Sigma$ and
$\partial N$ gives
\[
        \mathrm{II}_{\partial\Sigma\subset\Sigma}(X,Y)
        =
        \langle\widetilde\nabla_X\mu,Y\rangle
        =
        \mathrm{II}_{\partial N\subset N}(X,Y)=0.
\]
Thus $\partial\Sigma$ is weakly convex.

We now apply the boundary spectral comparison theorem of
\cite[Theorems 4 and Remark 22]{lijia2026} in which 
\(
        n=3,
        \theta=\frac43,
        (n-1)\lambda=\frac12,
\)
so $\lambda=1/4$.  Notice that \(
        \theta=\frac43\leq\frac{n-1}{n-2}=2.
\)
Since $n=3$, the exponent
$\frac{n-3}{n-1}\theta$ in the diameter estimate vanishes.  Hence
\[
        \operatorname{diam}_{\tilde g}(\Sigma)
        \leq\frac{\pi}{\sqrt{\lambda}}=2\pi, \quad  \Vol_{\tilde g}(\Sigma)
        \leq\lambda^{-3/2}\Vol(\Sph^3)
        =16\pi^2.
\]
\end{proof}

\begin{corollary}
\label{cor:weighted-relative-bubble-volume-dim-five}
When \(n=5\), the $4$-dimensional relative \(\mu\)-bubble in Theorem
\ref{thm:weighted-relative-bubble-dim-five} satisfies
\begin{equation}\label{eq:weighted-relative-volume-bound}
 \Vol_{\tilde g}(\Sigma)
 \leq
 \left(\frac{800}{43}\right)^2\Vol(\Sph^4_+)
 =
 \left(\frac{800}{43}\right)^2\frac{8\pi^2}{3}.
\end{equation}
\end{corollary}

\begin{proof} To obtain (\ref{eq:bubble-neumann-spectral-ricci}), drop the nonnegative boundary term in
\eqref{eq:weighted-bubble-spectral-ricci-dim-five} and divide by $\alpha$.
Then
\begin{equation}\label{eq:four-dimensional-neumann-form}
 \int_\Sigma
 \left(
 \frac{4}{(4-a)\alpha}|\nabla^\Sigma\psi|^2
 +\lambda_{\Ric}^\Sigma\psi^2
 \right)
 \geq
 \frac{\delta}{2\alpha}\int_\Sigma\psi^2.
\end{equation}
The first Neumann eigenfunction of the operator in
\eqref{eq:four-dimensional-neumann-form} is positive and satisfies the
corresponding pointwise spectral Ricci inequality.  Moreover,
$\partial\Sigma$ is totally geodesic, hence weakly convex.

In the notation of the boundary spectral volume comparison theorem
\cite[Theorem 4 and Remark 22]{lijia2026}, the dimension is $m=4$ and
\[
 \theta=\frac4{(4-a)\alpha}=\frac{43}{29}
 <\frac32=\frac{m-1}{m-2},
 (m-1)\lambda=\frac{\delta}{2\alpha}, \lambda=\frac{\delta}{6\alpha}=\frac{43}{800}.
\]
The volume comparison gives
\[
 \Vol_{\tilde g}(\Sigma)
 \leq
 \lambda^{-2}\Vol(\Sph^4),
\]
which is \eqref{eq:weighted-relative-volume-bound}. 
\end{proof}

\section{Almost linear volume growth and Proof of Main Theorem}
Let \(F:M^n\rightarrow \B^{n+1}\) be a complete two-sided stable free boundary minimal hypersurface for \(n\in \{4,5\}\). We may pass to the universal cover of \(M\) to obtain a simply-connected complete two-sided stable free boundary minimal hypersurface in \(B^{n+1}\), still denoted as \(M\). 
Then since \(n\geq 3\), \(N=M \setminus F^{-1}(0)\) is also simply connected.

Now  fix \(x\in \p M\subset \Sph^n\), consider the following exhaustion \(C_k=B^M_{kL}(x)\) of \(M\) for \(L=L_0+1\) where \(L_0\) is the constant defined in Theorem \ref{thm:weighted-relative-bubble-dim-five}.

If the universal cover \(M\) is compact, we may plug in (\ref{eq:fb-stability}) a non-zero constant function  to obtain a contradiction.
By Theorem \ref{OneEnd}, when the universal cover \(M\) is non-compact, it has only one end \((E_k)_{k\in\NN}\), which is a nonparabolic boundary end.

For some large \(k_0\) to be decided later, we decompose \(M\) as follows,
\[M=C_{k_0+i} \cup (\cup_{k=k_0+i}^{k_0+2i-1}M_k \cup P_k)\cup (E_{k_0+2i-1} \setminus C_{k_0+2i}),\]
where \(M_k=C_{k+1}\cap E_k\), \(P_k\) denotes the union of all components of \(E_{k-1} \setminus C_{k}\) besides \(E_k\) for \(k\geq k_0+1\). Note since \(M\) has only one end, each \(P_k\) is a finite union of compact components.

Within each  \(M_k=C_{k+1}\cap E_k\), we can find \(Y_k\), a band with boundary in the sense of section \ref{FB-Bubble}, such that 
\[\p Y_k\cap F^{-1}(0)=\emptyset, \quad d_{g}(\p_-Y_k, \p_+ Y_k)\geq L_0.\] 
Since \(M\) is simply connected,  we have for the nonparabolic end \(E_k=\overline{M\setminus C_k}\),  \(\p_1 E_k\) is connected, \(M_k\) is connected, 
and the above constructed \(Y_k\) is connected with connected \(\p_{\pm}Y_k\).
Furthermore \(\p_+Y_k\) separates \(\p_1 E_k\) and \(\p B^{M}_{kL+\epsilon_k}(x)\cap E_k\); \(\p_- Y_k\) separates \(\p_1 E_{k+1}\) and \(\p B^{M}_{(k+1)L-\epsilon_k'}(x)\cap E_k\), for some small \(\epsilon_k, \epsilon_k'>0\).

We now look at the band with boundary \(Y_k \setminus F^{-1}(0)\) after the Gromov-Lawson conformal change \(({Y}_k \setminus F^{-1}(0),\tilde{g})\), since \(r\leq 1\), \(\tilde{g}\geq g\), we have
\[
d_{\tilde{g}}(\p_-Y_k, \p_+ Y_k)\geq L_0.
\]
If \(Y_k \cap  F^{-1}(0) \neq \emptyset\), then the conformal change creates finitely many ends to \(Y_k\) since \(F^{-1}(0)\) is discrete, which we cut off with another separating boundary \(\p_1Y_k\) of finite components such that
\[
d_{\tilde{g}}(\p_1Y_k, \p_+ Y_k)\geq L_0.
\]
We may assume \(\p_- Y_k \cap \p_1 Y_k=\emptyset\).

We now apply Corollary \ref{cor:relative-bubble-estimates} and \ref{cor:weighted-relative-bubble-volume-dim-five} to obtain a \(\mu\)-bubble \(\tilde{\Sigma}_k'\) separating \(\p_- Y_k \cup \p_1 Y_k\) and \(\p_+ Y_k\); each component of  \(\tilde{\Sigma}_k'\) has volume bounded from above by a constant \(V_0>0\). In particular, we denote the component that separates \(\p_+ Y_k\) and \(\p_- Y_k\) as \(\tilde{\Sigma}_k\), whose non-empty boundary lies in \(\Sph^n\). 

 Using \(\tilde{\Sigma}_k\) is compact in \((N,\tilde{g})\), we denote \((\Sigma_k,g)\) to be \(\tilde{\Sigma}_k\) with the original metric; \(\Sigma_k \cap F^{-1}(0)=\emptyset\). 
 Since \(\p_-Y_k\) and \(\p_+ Y_k\) remain unchanged under \(\tilde{g}=r^{-2}g\), \(\Sigma_k\) separates \(\p_{1}E_k\) and \(\p_{1}E_{k+1}\).
We now want to obtain volume growth control of the end \((E_k)_{k\in \NN}\).

We first use the following free boundary version of curvature estimates that follows from the stable Bernstein theorem in \(\R^{n+1}\) for \(n\leq 5\). We recall that a compact manifold satisfies the weakly bounded geometry requirement.
\begin{proposition}[\cite{chodosh-li-stryker, mazet, wuyujie}  ]
    Let \((X^{n+1}, \p X)\) be a complete Riemannian manifold with weakly bounded geometry, and let \((M^n, \p M)\) be a complete stable minimal free boundary immersion whose second fundamental form is denoted as \(A\). Then, for \(n\leq 5\) there is a constant \(C(X,g)\) independent of \(M\) such that,
    \[
    \sup_{q\in M}|A(q)|\leq C(X,g)<\infty.
    \]
\end{proposition}
The proof for manifolds without boundary when \(n=3\) was given in Chodosh-Li-Stryker, and the proof for free boundary immersion was given in \cite{wuyujie}. Both proofs can be extended to dimension \(n\leq 5\) given the resolution of stable Bernstein problem in \cite{chodosh-li-stryker}\cite{mazet}.

Using Gauss-Codazzi equation, we have \(\Ric_M\geq -|A|^2 g\geq -(n-1)K^2g\) for some \(K \geq  0\).

\begin{lemma}
    If \((M^{n},\p M)\rightarrow (\B^{n+1},\Sph^n)\) for \(n\in\{4,5\}\) is a complete non-compact two-sided stable free boundary minimal hypersurface, let \((E_k)_{k\in\NN}\) denote the only nonparabolic end of \(M\) with respect to the exhaustion \(C_k=B^M_{kL}(x)\) for some \(x\in \p M\). Then the end \((E_k)\) has almost linear volume growth. More specifically, there is a constant \(C_0\) such that
    \[
    \sup_{k\in \NN}\Vol(M_k) \leq C_0.
    \]
\end{lemma}

\begin{proof}
Consider a minimizing geodesic \(\gamma_{\theta}(t)=\exp_p(t\theta)\) starting from \(p=x\) for some \(\theta\in S^+_pM=\{v\in T_p M, |v|=1, v\cdot \nu_{\p M}<0\}\), since \(M\) is convex, minimizing geodesic starting in the directions in \(S^+ M\) remains in the interior until it hits a cut time \(c(\theta)>0\). For \(0<t<c(\theta)\), we define the volume Jacobian so that \(d\mu_M=J(t,\theta)dt d\theta\), then the Bishop-Gromov Jacobian comparison shows that \(\frac{J(t,\theta)}{(\operatorname{sn}_K(t))^{n-1}}\) is non-increasing up to \(c(\theta)\). Therefore for \(0<t_1<t_2<c(\theta)\) we have
\[
\frac{J(t_2, \theta)}{J(t_1, \theta)}\leq \left(\frac{\operatorname{sn}_{K}(t_2)}{\operatorname{sn}_{K}(t_1)}\right)^{n-1}.
\]

Now we take \(k\geq k_0\) for some \(k_0\) to be decided, and we want to bound the volume of \(M_k\) by a fixed constant. For any \(s\in [kL,(k+1)L]\), and \(q\in \p B_{s}(p)\) there is a minimizing geodesic \(\gamma_{\theta}(s)\), and we denote \(t(\theta)\) to be the first time \(\gamma_{\theta}\) pass through the separating \(\mu\)-bubble \(\Sigma_{k-1}\), then \( (k-1)L\leq t(\theta) \leq kL\), and \(0\leq s-t(\theta)\leq 2L\). When \(q=\gamma_{\theta}(s)\) is not in the cut locus of \(p\), we apply the Jacobian comparison,
\[
\frac{J(s,\theta)}{J(t(\theta),\theta)}\leq \left(\frac{\operatorname{sn}_K(s)}{\operatorname{sn}_K(t(\theta))}\right)^{n-1} \implies \limsup_{k\rightarrow \infty} \frac{J(s,\theta)}{J(t(\theta),\theta)} =\frac{c_0}{2}<\infty.
\]
Choosing a large enough \(k_0\) we have \(J(s,\theta)\leq c_0 J(t(\theta),\theta)\).

If \(q=\gamma_{\theta}(s)\) is not in the cut locus of \(p\), then \(q'=\gamma_{\theta}(t(\theta))\) is also not in the cut locus of \(p\); denote the set of such \(q'\) as \(\hat{\Sigma}\). Then there is a smooth map from \(\Pi:\hat{\Sigma}\rightarrow \Theta_s \subset S^+_p M\) defined as \(\Pi(q')=\theta\), where \(\Theta_s\) is the set of points in \(S^{+}_p M\) such that \(\gamma_{\theta}(s)\) is not in the cut locus. Then 
\begin{align*}
   \int_{\Theta_s} J(t(\theta),\theta) d\theta  \leq  \int_{\hat{\Sigma}_k} J(t({\theta }),\Pi(q')) \operatorname{Jac}\Pi  =\int_{\Sigma_k}|\nu_{\Sigma_k}\cdot \gamma_{\theta}'| \leq \Vol(\Sigma_k)
\end{align*}

We  now have by area formula,
\begin{align*}
    \Vol(M_k)&\leq \int_{s=kL}^{(k+1)L} \Vol(\p B_{s}(p)) \\
    &= \int_{s=kL}^{(k+1)L}\int_{\Theta_s}J(s,\theta) d\theta ds\\
    &\leq c_0 \int_{s=kL}^{(k+1)L}\int_{\Theta_s} J(t(\theta),\theta) d\theta ds\\
    &\leq c_0L\Vol(\Sigma_k)\leq c_0V_0 L.
\end{align*}

This shows that for all \(k\geq k_0\), we have \[
\Vol(B_{(k+1)L}(p) \setminus B_{kL}(p))\leq c_0V_0L.
\]
So the almost linear volume growth of the end \((E_k)_{k\in \NN}\) follows.
\end{proof}

We can now finish the proof of the main theorem.
\begin{proof}
    Using the decomposition of \(M\), for \(k\geq k_0\), we find \(\rho_k\) a smooth function on \(\overline{M_k}\) with \(\operatorname{Lip}(\rho_k)\leq 2\), \(\rho_k\rvert_{\p_1 E_k}=kL\) and \(\rho_k \rvert_{\p M_k \setminus \p_1 E_k}=(k+1)L\) and extend this function continuously so that \(\rho_k \rvert_{P_{k}}=kL\) is constant.
    
    Take \(\phi_i\) to be the linear bump function that is equal to \(1\) within \(C_{k_0+i}\) and vanishes outside \(C_{k_0+2i}\). Define \(\varphi_k=\phi_i\circ \rho_k\) on \(M_k\cup P_k\) for \(k_0+i\leq k\leq k_0+2i-1\),  \(\varphi_k\) is also equal to \(1\) within \(C_{k_0+i}\) and vanishes outside \(C_{k_0+2i}\).

    We now plug \(\varphi_i\) into the stability inequality to obtain,
    \begin{align*}
        &\int_M (\Ric_X(\nu_M,\nu_M)+|A|^2)\varphi_i^2 +\int_{\p M}A^{\p X}(\nu_M,\nu_M)\varphi_i^2 \\
        &\leq \int_{M} |\nabla \varphi_i|^2= \sum_{k=k_0+i}^{k_0+2i}\int_{M_k} \frac{4}{(iL)^2}\leq (\frac{4C_0}{L^2}) \frac{1}{i}
    \end{align*}
    Letting \(i\rightarrow \infty\), the right-hand side of above inequality goes to zero, while the left hand side is bounded below by \(\int_{\p M}\varphi_i^2 \stackrel{i\rightarrow\infty}{\longrightarrow} \Vol(\p M)\), this is a contradiction.

\end{proof}

\section{Appendix: inradius bound}\label{appendix}
In this appendix we obtain an inradius bound for stable free boundary minimal hypersurfaces in a Riemannian manifold.
\begin{definition}
Let $k>0$.  The $k$-bi-Ricci curvature of a Riemannian manifold $X$ is defined
on ordered orthonormal pairs $v,w$ by
\[
        \BiRic_X^k(v,w)
        :=
        \Ric_X(v,v)+k\Ric_X(w,w)-R_X(v,w,w,v).
\]
We say that $X$ has nonnegative $k$-bi-Ricci curvature if
$\BiRic_X^k(v,w)\geq0$ for every ordered orthonormal pair $v,w$.
\end{definition}

\begin{lemma}\label{lem:weighted-inradius}
Let $3\leq n\leq5$, and let
\[
        \frac{n-1}{n}\leq k<\frac{4}{n-1}.
\]
Let $(X^{n+1},\partial X)$ have nonnegative $k$-bi-Ricci curvature.  Assume also
that
$
        h_{\partial X}\geq\kappa_X g_{\partial X}
$
for some $\kappa_X>0$.  Let
$
        F:(M^n,\partial M)\longrightarrow (X,\partial X)
$
be a complete noncompact two-sided stable free-boundary minimal hypersurface.
Then one has
\[
        d_M(p,\partial M)
        \leq
        \frac{c_{n,k}}{(n-1+k)\kappa_X}
        \qquad\text{for every }p\in M,
\]
where
\[
        c_{n,k}
        =
        n-1+\frac{(n-3)^2k}{4-(n-1)k}.
\]
In particular, when $n=5$, taking $k=4/5$ gives
        $d_M(p,\partial M)\leq 5/(3\kappa_X)$.
\end{lemma}

\begin{proof}
It is adaption of arguments from \cite{chengxufiniteindex,dengqintao}. For simplicity, we denote
$
        S=|A|^2+\Ric_X(\nu,\nu).
$
The free-boundary stability inequality is
\[
        \int_M |\nabla\varphi|^2\,d\mu
        \geq
        \int_M S\varphi^2\,d\mu
        +\int_{\partial M}h_{\partial X}(\nu,\nu)\varphi^2\,d\sigma .
\]
Since $M$ is noncompact,  we know there exists a positive function $u$ satisfying
\begin{equation}\label{eq:weighted-robin-solution}
        -\Delta u-Su=0
        \quad\text{in }M,
        \qquad
        \partial_\eta u=h_{\partial X}(\nu,\nu)u
        \quad\text{on }\partial M.
\end{equation}
Denote $w=\log u$ and consider
\[
        \hat g=u^{2k}g .
\]
Let $\hat\gamma$ be a $\hat g$-minimizing segment from $p$ to $\partial M$.
Parametrize this curve by the original arclength $s\in[0,\ell]$, so that
$\gamma(0)=p$ and $\gamma(\ell)\in\partial M$.  By the first variation, the
endpoint is orthogonal to $\partial M$, hence $\gamma'(\ell)=\eta$.  Choose a
$g$-orthonormal frame $e_1,\ldots,e_n$ along $\gamma$, with $e_1=\gamma'$ and
$e_2,\ldots,e_n$ tangent to $\partial M$ at $s=\ell$.  Then
$\hat e_i=u^{-k}e_i$ is $\hat g$-orthonormal.

The second variation of $\hat g$-length for variations fixed at $p$ and free on
$\partial M$ gives, for every smooth $\phi$ with $\phi(0)=0$,
\begin{equation}\label{eq:weighted-hat-second-variation}
        (n-1)\int_0^{\hat\ell}
        \left(\frac{d\phi}{d\hat s}\right)^2\,d\hat s
        \geq
        \int_0^{\hat\ell}
        \widehat{\operatorname{Ric}}(\hat e_1,\hat e_1)\phi^2\,d\hat s
        +\hat H_{\partial M}\phi(\hat\ell)^2
\end{equation}
Since $d\hat s=u^kds$, the left-hand
side is
\[
        (n-1)\int_0^\ell u^{-k}(\phi')^2\,ds .
\]
The usual conformal Ricci formula for $\hat g=e^{2kw}g$ is
\begin{align*}
        \widehat{\operatorname{Ric}}(\hat e_1,\hat e_1)
        &=
        u^{-2k}\bigl(
        \operatorname{Ric}^M(e_1,e_1)
        -(n-2)k\nabla^2w(e_1,e_1)-k\Delta w        \\
        &\qquad
        +(n-2)k^2(w')^2-(n-2)k^2|\nabla w|^2
        \bigr).
\end{align*}
Along a $\hat g$-geodesic, the quadratic terms cancel with the difference
between $\nabla^2w(e_1,e_1)$ and $w''$.  Indeed,
\[
        \nabla_{e_1}e_1=k(\nabla w-w'e_1),
        \qquad
        \nabla^2w(e_1,e_1)=w''-k|\nabla w|^2+k(w')^2 .
\]
Thus the conformal Ricci formula along $\gamma$ gives
\[
        \widehat{\operatorname{Ric}}(\hat e_1,\hat e_1)
        =
        u^{-2k}\left(
        \operatorname{Ric}^M(e_1,e_1)
        -(n-2)kw''
        -k\Delta w
        \right),
\]
and the conformal change of boundary mean curvature gives
\[
        \hat H_{\partial M}
        =
        u^{-k}\left(H_{\partial M}+(n-1)k\partial_\eta w\right).
\]
Using \eqref{eq:weighted-robin-solution}, namely
\[
        \Delta w+|\nabla w|^2=-S,
        \qquad
        \partial_\eta w=h_{\partial X}(\nu,\nu),
\]
the second variation becomes
\begin{align}
        (n-1)\int_0^\ell u^{-k}(\phi')^2\,ds
        &\geq
        \int_0^\ell
        \left(\operatorname{Ric}^M(e_1,e_1)+kS
        -(n-2)kw''+k|\nabla w|^2\right)\phi^2u^{-k}\,ds
        \notag\\
        &\quad
        +\left(H_{\partial M}
        +(n-1)kh_{\partial X}(\nu,\nu)\right)u^{-k}\phi(\ell)^2 .
        \label{eq:weighted-second-variation}
\end{align}
Now Denote $\phi=fu^{k/2}$, where $f(0)=0$.  Then
\[
        u^{-k}(\phi')^2
        =
        (f')^2+kff'w'+\frac{k^2}{4}f^2(w')^2
\]
where $w'=(\log u)'$, and
\[
        \int_0^\ell f^2w''\,ds
        =
        f(\ell)^2h_{\partial X}(\nu,\nu)
        -2\int_0^\ell ff'w'\,ds .
\]
Substituting into \eqref{eq:weighted-second-variation} gives
\begin{align}
        &(n-1)\int_0^\ell (f')^2\,ds
        +\frac{(n-1)k^2}{4}\int_0^\ell f^2(w')^2\,ds
        +(3-n)k\int_0^\ell ff'w'\,ds        \notag\\
        &\qquad\geq
        \int_0^\ell f^2
        \left(\operatorname{Ric}^M(e_1,e_1)+kS
        +k|\nabla w|^2\right)\,ds
        +\left(H_{\partial M}
        +kh_{\partial X}(\nu,\nu)\right)f(\ell)^2 .
        \label{eq:weighted-expanded}
\end{align}
Since $|\nabla w|^2\geq(w')^2$, we obtain
\begin{align}
        &(n-1)\int_0^\ell (f')^2\,ds
        -\frac{k(4-(n-1)k)}4\int_0^\ell f^2(w')^2\,ds
        +(3-n)k\int_0^\ell ff'w'\,ds        \notag\\
        &\qquad\geq
        \int_0^\ell f^2
        \left(\operatorname{Ric}^M(e_1,e_1)+kS\right)\,ds
        +\left(H_{\partial M}
        +kh_{\partial X}(\nu,\nu)\right)f(\ell)^2 .
        \label{eq:weighted-pre-cauchy}
\end{align}
By Cauchy's inequality,
\[
        -\frac{k(4-(n-1)k)}4a^2+(3-n)kab
        \leq
        \frac{(n-3)^2k}{4-(n-1)k}b^2 .
\]
Therefore
\begin{equation}\label{eq:weighted-boundary-length}
        c_{n,k}\int_0^\ell (f')^2\,ds
        \geq
        \int_0^\ell f^2
        \left(\operatorname{Ric}^M(e_1,e_1)+kS\right)\,ds
        +\left(H_{\partial M}
        +kh_{\partial X}(\nu,\nu)\right)f(\ell)^2 .
\end{equation}
The interior term is nonnegative by standard estimate.  Indeed, for any unit $v=e_1\in T_pM$,
\[
        \operatorname{Ric}^M(v,v)+kS
        =
        \BiRic_X^k(v,\nu)
        +k|A|^2-|A(v,\cdot)|^2 .
\]
Since $A$ is trace-free,
\[
        |A(v,\cdot)|^2\leq\frac{n-1}{n}|A|^2,
\]
and hence the last expression is nonnegative by the assumptions
$\BiRic_X^k\geq0$ and $k\geq(n-1)/n$.

At the endpoint, the free-boundary condition gives
\[
        H_{\partial M}+kh_{\partial X}(\nu,\nu)
        =
        \sum_{i=2}^n h_{\partial X}(e_i,e_i)
        +kh_{\partial X}(\nu,\nu)
        \geq (n-1+k)\kappa_X .
\]
Taking $f(s)=s$ in \eqref{eq:weighted-boundary-length} gives
\[
        c_{n,k}\ell
        \geq
        (n-1+k)\kappa_X\ell^2,
\]
and hence $\ell\leq c_{n,k}/((n-1+k)\kappa_X)$.  The original distance from
$p$ to $\partial M$ is bounded by the original length $\ell$ of this curve.
\end{proof}

The following two consequences are of particular interest.
\begin{corollary}\label{cor:mean-convex-inradius}
Let $3\leq n\leq4$.  Let $(X^{n+1},\partial X)$ have nonnegative bi-Ricci
curvature and assume that
$
        H_{\partial X}\geq H_0
$
for some $H_0>0$, where $H_{\partial X}$ is the trace mean curvature.  Let
$
        F:(M^n,\partial M)\longrightarrow (X,\partial X)
$
be a complete noncompact two-sided stable free-boundary minimal hypersurface.
Then
\[
        d_M(p,\partial M)
        \leq
        \frac{c_n}{H_0}
        \qquad\text{for every }p\in M,
\]
where
\[
        c_n=n-1+\frac{(n-3)^2}{5-n}.
\]
In particular, $M$ has no interior ends.
\end{corollary}

\begin{proof}
Take $k=1$ in the proof of Lemma \ref{lem:weighted-inradius}.  Then
$\BiRic_X^1=\BiRic_X$.  In the last step of the proof of previous theorem, we now have
\[
        H_{\partial M}+h_{\partial X}(\nu,\nu)
        =
        H_{\partial X},
\]
so the pointwise lower bound on the principal curvatures of $\partial X$ can be
replaced by the trace lower bound $H_{\partial X}\geq H_0$.
\end{proof}

\begin{corollary}
\label{cor:strict-inradius-unit-ball}
Let
$
        F:(M^4,\partial M)\longrightarrow
        (\overline{\B^5},\Sph^4)
$
be a complete noncompact two-sided stable free-boundary minimal immersion.
Then
\begin{equation}\label{eq:strict-inradius-unit-ball}
        d_M(p,\partial M)\leq\frac{32}{35}<1
        \qquad\text{for every }p\in M.
\end{equation}
In particular, $F(M)$ does not pass through the origin, and $M$ has no interior
ends.
\end{corollary}

\begin{proof}
The Euclidean ball has $\BiRic^{3/4}=0$, and its boundary satisfies
$h_{\Sph^4}=g_{\Sph^4}$.  Applying Lemma
\ref{lem:weighted-inradius} with $n=4$, $k=3/4$, and $\kappa_X=1$ gives
\[
 d_M(p,\partial M)
 \leq
 \frac{3+\dfrac{3/4}{4-3(3/4)}}{3+3/4}
 =\frac{32}{35}<1.
\]
If $F(p)=0$, every curve from $p$ to $\partial M$ is mapped to a curve in
$\overline{\B^5}$ joining the origin to $\Sph^4$.  Since its Euclidean length,
and hence its induced $g$-length, is at least $1$, this would imply
$d_M(p,\partial M)\geq1$, a contradiction.
\end{proof}

\bibliographystyle{alpha}
\bibliography{reference}
\end{document}

%% file: head.tex
\usepackage{amsmath}
\usepackage{amssymb}
\usepackage{amsthm}
\usepackage{hyperref}
\usepackage[noabbrev,capitalize]{cleveref}
\usepackage{mathrsfs}
\usepackage{mathtools}
\usepackage{slashed}
\usepackage{tikz-cd}
\usepackage[shortlabels]{enumitem}

\setenumerate{label=(\roman*)}

\DeclareMathOperator{\Vol}{Vol}

\DeclareMathOperator{\Ric}{Ric}
\DeclareMathOperator{\biRic}{biRic}

\newcommand{\Int}[1]{\mathring{#1}}

\newcommand{\cH}{\widetilde{H}}

\def\sideremark#1{\ifvmode\leavevmode\fi\vadjust{\vbox to0pt{\vss
 \hbox to 0pt{\hskip\hsize\hskip1em
 \vbox{\hsize3cm\tiny\raggedright\pretolerance10000
 \noindent #1\hfill}\hss}\vbox to8pt{\vfil}\vss}}}

\newtheorem{theorem}{Theorem}[section]
\newtheorem{proposition}[theorem]{Proposition}
\newtheorem{lemma}[theorem]{Lemma}
\newtheorem{corollary}[theorem]{Corollary}

\theoremstyle{definition}
\newtheorem{definition}[theorem]{Definition}

\theoremstyle{remark}

\numberwithin{equation}{section}